%% file: Flags_arxiv.tex
\documentclass[10pt,a4paper,english]{article}
\usepackage{amssymb,amsmath}
\usepackage{amsfonts}
\usepackage{eucal}
\usepackage{amsthm}
\usepackage{enumerate}
\usepackage{macros-perso}
\usepackage[all,dvips]{xy}
\CompileMatrices

\newtheorem{defn}{Definition}
\newtheorem{thm}{Theorem}
\newtheorem{cor}{Corollary}
\newtheorem{lemm}{Lemma}
\newtheorem{prop}{Proposition}
\newtheorem{rem}{Note}[section]

\def\ra{\mathbf{\Lambda}}

\def\hn{\mathrm{HN}}
\def\RH{\mathrm{RH}}

\def\ab{\mathrm{ab}}
\renewcommand{\hsg}{\ti{\sg}}
\author{Eduardo Corel\footnote{E-mail : ecorel@gwdg.de} and Elie
Compoint\footnote{E-mail: compoint@math.univ-lille1.fr}}
\title{Stable Flags and the Riemann-Hilbert Problem}
\date{\today}

\begin{document}
\maketitle

\begin{abstract}
We tackle the Riemann-Hilbert problem on the Riemann sphere
as stalk-wise logarithmic modifications of the classical R\"ohrl-Deligne vector bundle.
We show that the solutions of the Riemann-Hilbert problem are in
bijection with some families of local filtrations which are stable
under the prescribed monodromy maps. 
We introduce the notion of Birkhoff-Grothendieck trivialisation, and
show that its computation corresponds to geodesic paths in some local 
affine Bruhat-Tits building. We use this to compute how the type of a
bundle changes under stalk modifications, and give several
corresponding algorithmic procedures.
\end{abstract}




\section*{Introduction}
\label{in}

The Riemann-Hilbert problem (RHP) has a long and distinguished history, not
even devoid of suspense, for it has been solved several times, using
different tools, in a seemingly complete and positive way. It is
finally A. A. Bolibrukh, in a celebrated series of papers at the
beginning of the 1990's who clarified the situation, by rigorously
defining (and exhibiting a counter-example) to the strongest version
of the RHP, thereby showing that people before him had either
committed a mistake, or solved in reality a weaker problem.

The modern approach to the RHP was initiated by H. R\"ohrl in the
1950's who used
the theory of vector bundles in a way that has been conserved
since. First, one constructs a vector bundle $\E$ outside the singular
points, whose cocycle mimicks the monodromy. We call this the {\em
topological} RH problem, since the monodromy is so much encoded in the
topology of the constructed bundle, that construction of the required
connection becomes essentially trivial.
The second step consists in extending the bundle (and the connection)
to the singular points by means of a {\em local solution} to the
inverse monodromy problem. It has been exposed in great generality in
P. Deligne's work (\cite{De}) how to extend a holomorphic vector
bundle $\E$, defined over the complement of a divisor $D$ and
endowed with a holomorphic connection $\na$ having a prescribed
monodromy about $D$, into a logarithmic connection
$(\ov{\E},\ov{\na})$ with singularities on the divisor, uniquely
determined by  a section of the natural projection $\app{\C}{\C/\Z}$.
In this way, we get all logarithmic extensions of $\E$ with {\em
non-resonant}
residue (the {\em Deligne lattices}).  These two steps are sufficient
to solve positively the weak Riemann-Hilbert problem ({\it i.~e.} with
regular singularities). Note, however, that in this second level, two
different types of problems have been mixed. The
connection constructed is essentially unique up to {\em meromorphic}
equivalence (to be rigorously defined later) whereas the holomorphic
vector bundle setting already introduces much finer {\em holomorphic}
equivalence problems. This fact can contribute to explain some of the
confusions that have surrounded the precise formulation of the RHP.

The {\em strong Riemann-Hilbert problem} asks however for a
logarithmic bundle (with the prescribed monodromy) which is moreover
{\em trivial}.

So, to solve the Riemann-Hilbert problem in this way, one must modify the
constructed Deligne bundle, over the support of the singular divisor
exclusively (to keep the singular set invariant),
while conserving its logarithmic character, until a trivial bundle
is eventually found. Until A. Bolibrukh's celebrated counter-example
(\cite{Bo1}), it was widely acknowledged that this was possible, and
it is indeed so in
several ``generic'' instances, although some mistakes in the seemingly
general solution by Plemelj had already been pointed out.

The counter-example found by A. A. Bolibrukh to the strong
Riemann-Hilbert problem requires the knowledge of {\em all} the
logarithmic extensions of a regular connection, in order to prove that
none is trivial. Despite the production of both counter-examples and
sufficient
conditions for a positive answer, no general necessary and sufficient
conditions for
the solubility of the strong Riemann-Hilbert problem have been given
in terms of the
monodromy representation only.

As already stated, the strong Riemann-Hilbert problem admits a solution
if and only if the stalks of the Deligne bundle over the singular set
can be replaced by logarithmic lattices in such a way that the resulting
bundle is trivial. To tackle this problem, it is logical to
\begin{enumerate}[a)]
\item determine the set of all logarithmic lattices above a given point,\label{c1}
\item  get a criterion for the triviality of the modified bundle.\label{c2}
\end{enumerate}

In this paper, we solve problem~\ref{c1} by giving a complete description
of the logarithmic lattices in terms of flags stabilised under the action of the 
residue of the connection. After finding such a characterisation, we became aware
that such a description appears in~\cite{Sa} p., who attributes this result
to Deligne-Malgrange. However, we preferred to state the whole result in the 
geometric terms of our paper. We also give a partial answer to problem~\ref{c2}.
In the case of $\P$, the type of a vector bundle gives such a triviality criterion.
In our selected approach, starting with the Deligne bundle $D$, we perform a 
modification of a finite number of stalks, resulting in the bundle $D^{\mathrm{mod}}$.
The question is then to compute the type of the modified bundle $D^{\mathrm{mod}}$.
Generalising a result by Gabber and Sabbah (proposition~\ref{GS}), we 
show how to determine the type of $D^{\mathrm{mod}}$ from the type of $D$.
Thus, problem~\ref{c2} is reduced to computing the type of the Deligne bundle.
In a second step, we show that this problem in turn is reduced to the 
well-known problem of {\em connection matrices}.

The paper is organised as follows.

In a first section, we define the category in which we will work, and what we precisely
mean by ``modifying a bundle in one or several points''. In a second part,
we describe the geometry on the local lattices involved. We describe this 
geometry in terms of the affine Bruhat-Tits building of $\SL_n$. This choice
is justified by the fact that, more than the local lattices themselves, our
description relies on the homothety class of such lattices, which are
precisely the vertices of the considered Bruhat-Tits building. Several
invariants attached to the underlying 1-skeleton, particularly the 
natural graph-theoretical distance, will play an important paper.

In a third part, we use this setting to give an effective method to compute 
how the type of an arbitrary bundle $E$ is modified under a modification 
$E^{\mathrm{mod}}$ of $E$. This algorithm can also be applied to
compute the type of the bundle $E$. This third section concludes
with a generalisation of an essential result due to Bolibrukh, the
permutation lemma, for which we provide an interesting geometric 
interpretation.

The fourth section gives the complete description of the set of
logarithmic lattices in terms of flags which are stable under the
action of the residue of the connection on the Deligne bundle.

In the last part, after recalling the construction of the 
classical R\"ohrl-Deligne bundle, we give a very concise proof
of Plemelj's theorem on the Riemann-Hilbert solubility. This
well-known result becomes an immediate consequence of the 
geometrical interpretations of the permutation lemma and the
set of logarithmic lattices. We describe all trivialisations
of the Deligne bundle over an arbitrary point, and we establish a
stronger inequality on the type of the weak solutions to Riemann-Hilbert 
in the irreducible case. Finally, we give algorithmically effective
procedures that allow to search the space of weak solutions. 

\section{Holomorphic Vector Bundles}
\label{HVB}

Let $X$ be a compact Riemann surface and let $\appto{\pi}{E}{X}$ be
a rank $n$ holomorphic vector bundle. The sheaf $\E$
of holomorphic sections of $E$ is a locally free sheaf of
$\ox$-modules of the same rank $n$, where $\ox$ denotes as usual 
the sheaf of holomorphic functions on $X$. There is a well-known
equivalence between these two
categories. However, it is also well known that this equivalence
fails for sub-objects of the same rank. Any locally free subsheaf
$\F\subset\E$ of $\ox$-modules of rank $n$
on $X$ can be seen as a sheaf locally generated over $\ox$ by {\em
holomorphic} sections of $E$,
while the equivalence of categories allows us to call $\F$ a
holomorphic {\em vector bundle}. However, it
is not possible to find an equivalent to $\F$ as a sub-bundle of $E$
since both have the same rank.

\paragraph{Meromorphic Connections.}
Let $\div=\sum_{i=1}^p m_i x_i$
be a positive divisor on $X$. Let $\h_{\div}$ be the sheaf of
meromorphic functions on $X$ having pole orders bounded by $D$
({\em i.e.} less than $m_i$ at $x_i$). Let
$\dd_{\div}=\{x_1,\ldots,x_p\}$ be  its {\em support}.  For any finite
set $S=\{y_1,\ldots,y_t\}$, let $[S]=y_1+\cdots+y_t$.

Let $\na:\app{\E}{\tens{\E}{\ox}{\omdiv{\div}}}$ be a meromorphic
connection with singular divisor $\div$ on a vector bundle $\E$ of
rank $n$. In the sequel, we will always assume that $\div$ is the
smallest possible. Sometimes for simplicity we'll just say
``connection'' for the pair $(\E,\na)$. The {\em Poincar\'e rank} of
$\na$ at $x\in X$  is the integer $\p_{x}(\na)=\max(0,m_x-1)$. If
$\p_{x}(\na)=0$, the sheaf $\E$ is said to be {\em logarithmic with
  respect to $\na$ at $x$}. Let $\dd=\abs{\div}$ be the singular, and
$\dd_{\log}=\{x\in \dd \tq \p_x(\na)=0\}$ the logarithmic singular
sets of $\na$. If $\dd_{\log}\neq \emptyset$, then one can define the
residue map $\r \na\in \End(\E/\E_{-[\dd_{\log}]})$, where
$[\{x_1,\ldots,x_m\}]=x_1+\cdots+x_m$. If $x\in
\dd\backslash\dd_{\log}$ is not logarithmic, the residue of
$z^{\p_x(\na)}\na$ for any local coordinate $z$ induces the
well-defined {\em polar map} $PM\na \in \mathbb{P}
\End(\E/\E_{-[\dd]})$ of $\na$ over $\E$. We will specify in parentheses
the bundle if necessary.

\paragraph{The Meromorphic Bundle.}
Let $\V=\tens{\E}{\ox}{\mx}$ be the sheaf of meromorphic sections of
$E$. A meromorphic connection $\na$ on $\E$ induces a canonical extension to $\V$.
Since the sheaf $\E$ can be embedded
into $\V$, we consider henceforth only the set $$H=\{\F\subset \V \tq
\F\stackrel{{\mathrm{loc.}}}{\simeq}\ox^n\}$$
of {\em holomorphic vector bundles} of $\V$. Each such bundle $\F$ is 
automatically endowed with a meromorphic connection induced by $\na$. By simplicity,
we won't make any notational difference between all these connections. 

We say that $\F\in H$ is {\em trivial} if $\F\simeq \ox^n$, or, equivalently, if $\F$ is generated by
its global sections. In this case, the set $\Gamma(X,\F)$ of global
sections is a $\C$-vector space
of dimension $n$, admitting as basis vectors global meromorphic
sections of $E$ with specific constraints on their divisor.
Let $H_0\subset H$ be the subset of trivial holomorphic bundles in $\V$.
The following result is well known (\cite{Sa}, p.?)

\begin{lemm}\label{triv1}
Let $\F\in H$ be a holomorphic vector bundle in $(\V,\na)$. If $\F$ is trivial,
then the space $Y_{\F}=\Gamma(X,\F)$ of global sections is a $\C$-vector space of dimension $n$.
For any logarithmic singularity $s\in \dd_{\log}(\F)$, the residue $\r_s^{\F}\na$ induces a
well-defined endomorphism $\psi_s\in \End_{\C}(Y_{\F})$.
\end{lemm}

\paragraph{Stalks and Lattices.}
For any $x\in X$, the stalk $\F_x$ of a holomorphic
vector bundle is a free $(\ox)_x$-submodules of rank $n$ (or {\em lattice})
of the stalk $V_x=\V_x$, which is a vector space of dimension $n$
over the fraction field $K_x$ of $(\ox)_x$. Let $\ra_x$ be the set of
lattices of $V_x$. We define an equivalence relation $R_x$ on $H$ as
$$(\F,\ti{\F})\in R_x\mbox{ if and only if }\F_{\vert
X\backslash\{x\}}=\ti{\F}_{\vert X\backslash\{x\}}.$$
For simplicity, we will drop the index $x$ as soon as no ambiguity
can be feared. Any coset of $H/R_x$ can be identified with the set
$\ra_x$, by identifying a vector bundle $\E$ in a given coset of
$H/R_x$ with its stalk $\E_x\in\ra_x$ at $x$.

Actually, since $X$ is compact, two vector bundles $\E,\F\in H$ have
equal stalks outside a finite set $\Delta(\E,\F)$.

\begin{lemm}
\label{coco}
Let $\E\in H$ be a holomorphic vector bundle. For any
family of lattices $M_x\in \ra_x$ for $x$ in a discrete set $\dd$,
there exists
a unique vector bundle $\E^M\in H$ such that
$$(\E^M)_x=\left\{ \begin{array}{l}
\E_x \textrm{ if } x \not\in \dd\\
M_x\textrm{ if } x \in \dd
\end{array}\right.$$
Moreover, for any $\F\in H$, there exists such a discrete set $\dd$ and
a family $(M_x\in \ra_x)_{x\in \dd}$ of lattices such that $\F=\E^M$.
If $\E$ is endowed with a meromorphic connection $\na$, there is
a canonical extension $\na^M$ of $\na_{\vert X\backslash \dd}$ 
as a meromorphic connection  on $\E^M$.
\end{lemm}

The sheaf $\V$ is always trivial, and the group $G$ of (meromorphic) 
automorphisms of the space
$\Gamma(X,\V)$ is isomorphic to $\G(\C(X))$. Let $\ra_x^0=\ra_x\cap
H_0$ be the
set of trivial bundles in the coset $\ra_x$. The subgroup $G_x\subset G$ of
automorphisms of $\Gamma(X,\V)$ that leave $\ra_x^0$ globally invariant is
called the {\em group of monopole gauge transforms at $x$}. Each element
of $G_x$ sends a trivial sheaf $\F$ to a trivial sheaf $\ti{\F}$ such that
 $\F_{\vert X\backslash\{x\}}=\ti{\F}_{\vert X\backslash\{x\}}$. 
An element of $G_X$ modifies at most the stalk $\F_{x}$.

\section{Lattices and the Affine Building of $\SL_n$}
\label{sdc}

In this section, we fix a point $x\in X$ and a coset $\ra\in H/R_x$. We
drop the index $x$ for simplicity. The field $K$ is local, and endowed
with the discrete valuation $v=\o_x$, whose valuation ring is $\h$,
whose maximal ideal is $\id$ and residue field $k=\C$. As already mentioned,
$V$ is a $K$-vector space of dimension $n$.

\paragraph{Flags.}
Here assume that $V$ is a vector space over an arbitrary field 
of characteristic 0.
Given a flag $\F$ of vector spaces $0=F_0\subset F_1 \subset \cdots
\subset F_s = V$ in $V$,
where $s=\abs{\F}$ is the {\em length} of the flag, the {\em
signature} of $\F$ is the integer sequence $(n_1,\ldots,n_s)$ where
$n_i=\dim\left(F_i/F_{i-1}\right)$. 
A map $u\in \en_K(V)$ is said to {\em stabilise} the flag $\F$ if
$u(F_i)\subset F_i$ for all $0\lsl i\lsl \abs{\F}$. Let $\fl(V)$ for
the set of flags of $V$, and $\fl_u(V)$ for the subset of flags that
are stabilised by $u$. Recall that a flag 
$\F':0=F'_0\subset F'_1 \subset \cdots \subset F'_s = V$ is 
said to be {\em transversal} to $\F$ if $F'_i\oplus F_{s-i}=V$ for 
$1\lsl i\lsl s$. Note that in this case, the signature of $\F'$ is
then equal to $(n_s,\ldots,n_1)$.
For any subspace $M\subset V$, let 
$\F\cap M$ denote the flag of $M$ composed of the {\em distinct} subspaces
among the $F_i\cap M$.
An {\em $\F$-admissible sequence} is an integer sequence 
$$u=(\underbrace{k_1,\ldots,k_1}_{n_1\textrm{ times}},
\underbrace{k_2,\ldots,k_2}_{n_2\textrm{ times}},\ldots,
\underbrace{k_s,\ldots,k_s}_{n_s\textrm{ times}})\mbox{ with
}k_1<\cdots<k_s.$$
Let $\Z^n(\F)$ be the set of integer $\F$-admissible sequences, and 
let $$W(V)=\{(\F,\k) \tq  \F \in \fl(V)\mbox{ and }\k\in\Z^n(\F)\}$$ be 
the set of {\em admissible pairs} of $V$. For any integer sequence $\k=(k_1,\ldots,k_n)$, 
let $$\Delta \k=\max_{i}k_i-\min_{i} k_i\mbox{ and }i(\k)=\sum_{j=1}^n(\max_{i}k_i-k_j).$$

\paragraph{Lattices.}
Let $L(u)$ denote the free $\h$-module spanned by a family $(u)$ of
vectors in $V$. An $\h$-module $M\in V$ is a lattice if there exists a
$K$-basis $(e)$ of $V$ such that $M=L(e)$.
Let $$v_{\la}(x)=\min\{k\in \Z \tq \id^kx\in \la\},$$
be the natural valuation of $V$ induced by $\la$. For any lattices
$M\subset \la$ in $V$, we define the {\em interval} $[M,\la]$ as $$[M,\la]=\{N\in\ra \tq M\subset
N\subset \la\}.$$

\paragraph{Elementary Divisors.}
Let $z$ be a uniformising parameter of $K$. For any two lattices $\la$ and
$M$ of $V$, there exists a unique
increasing sequence of integers $k_1 \leqslant \cdots \leqslant
k_n$ (the \emph{\del\! of $M$ in $\la$}) and an $\h$-basis
$(e_1, \ldots,e_n)$ of $\la$ such that $(z^{k_{1}}e_1, \ldots
,z^{k_n}e_n)$ is a basis of $M$. Such a basis $(e)$ is called a {\em Smith basis of $\la$ for $M$}. 
We will write them $k_{i,\la}(M)$ if we want to specify the
respective lattices, and we write $$\k_{\la}(M)~=~(k_{1,\la}(M), \ldots ,k_{n,\la}(M)).$$
Note that $k_{1,\la}(M)=v_{\la}(M)$, and let $M_{\la}=z^{-v_{\la}(M)}M$.
It is convenient
to be a bit more lax on the definition, and allow the \del\! to appear in another order.

The subgroup of the lattice stabiliser $\Go$ 
that acts on the set of Smith bases of $\la$ for $M$ is the {\em lattice $\k$-parabolic subgroup}
$\mf{G}_{\k}$, whose intersection with $\Gc$ is the $\k$-parabolic group $G_{\k}$ defined as
$$\mf{G}_{\k}=\{P\in \Go\tq v(P_{ij})\gsl k_i-k_j\}\mbox{ and }G_{\k}=\{P\in \Gc\tq P_{ij}\neq 0 \Rightarrow k_i\lsl k_j\}.$$
We will say for short that $P\in\mf{G}_{\k}$ (resp. $P\in G_{\k}$) is $\mf{G}_{\k}$-parabolic (resp. $\k$-parabolic).
For any sequence $\k=(k_1,\ldots, k_n)$, let $z^{\k}$ for the diagonal matrix
$\diag(z^{k_1},\cdots, z^{k_n})$. We will frequently use the following type of
diagram
$$
\xymatrix{
\la:(e) \ar[r]^-{z^{\k}} \ar[d]^-{P} & M:(\vep) \ar[d]^-{\tilde{P}} \\
\la:(e') \ar[r]^-{z^{\k}} & M:(\vep')
}
$$
which means that $(e)$ and $(e')$ are two Smith bases of $\la$ for $M$. Note
that in this case $P\in\Go$ is $\mf{G}_{\k}$-parabolic (and $\ti{P}$ is $\mf{G}_{-\k}$-parabolic, symmetrically).

Sometimes, we will find it more
convenient to consider the \del\! with their multiplicities. In this
case, we will put $\kk_1,\ldots,\kk_s$ for the distinct \del\! of $M$
in $\la$ and let $n_i$ be their respective multiplicities. The set
$[n]=\{1,\ldots,n\}$ of indices of ordinary ({\em simple}) \del\! is
partitioned into the subsets $I_j$ corresponding to a single value of
the \del\! $$I_j=\{1\lsl \ell \lsl n \tq k_{\ell}=\kk_j\}\textrm{ for
}1\lsl j \lsl s.$$

\subsection{The Affine Building of $\SL(V)$}
\label{build}

For this section, which is standard, good references are~\cite{Br,G}.
The affine building $\B_n$ naturally attached to $\SL(V)$ is the
following $(n-1)$-dimensional
simplicial complex. Two lattices $\la$ and $M$ are {\em homothetic} if
there exists $\alpha \in K^*$ such that $M=\alpha \la$. Let $[\la]$ be
the homothety class of the lattice $\la$ in $V$. The vertices of
$\B_n$ are the homothety classes of lattices in $V$, and an edge
connects two vertices $L$ and $\pr{L}$ if and only if there exist
representatives $\la$ of $L$ and $M$ of $\pr{L}$ such that
$\id\la\subset M \subset \la$. The affine building $\B_n$ is the
{\em flag simplicial complex} associated with this graph, or in other
terms, its {\em clique complex}. A maximal simplex, or {\em chamber} in
$\B_n$, is an $n$-chain of vertices $L_0,\ldots,L_{n-1}$ with
representatives $\la_i$ for $0\lsl i \lsl n-1$ satisfying
$$\id\la_0\subset \la_1\subset \cdots \subset \la_{n-1}\subset \la_0.$$
The natural {\em graph-theoretic distance} $d$ in $\B_n$, that is
the length of the shortest path between two lattice classes $L$ and $L'$, 
and the {\em index} of $L'$ with respect to $L$, are given by
\begin{equation}\label{d-i}
d(L,\pr{L})=k_{n,\la}(M)-k_{1,\la}(M)\mbox{ and }[L:L']=\sum_{i=1}^n\left(k_{i,\la}(M)-k_{1,\la}(M)\right)
\end{equation}
for any representatives $\la,M$ of $L,\pr{L}$. Note that $d(L,\pr{L})=-v_{\la}(M)-v_M(\la)$ also holds.

\paragraph{Geodesics.}
A {\em geodesic} is a path $\Gamma$ in $\B_n$ such that for any vertices $L,L'\in \Gamma$,
the length of the path between $L$ and $L'$ induced by $\Gamma$ is equal to
$d(L,L')$. The following result explains how to construct a geodesic algebraically.

\begin{prop}\label{geo}
Let $L,L'\in \B_n$, and let $d=d(L,L')$.
For $k\in \N$, let $L_k=[\la'+\id^k\la]$, where $\la\in L$ and $\la'\in L'$ 
are such that $v_{\la}(\la')=0$. Then $d(L_k,L_{k+1})=1$
for $0\lsl k\lsl d-1$ and $L_d=L'$.
The path $$\Gamma(L,L')=(L_0,L_1,\ldots,L_d)$$ is called the {\em
geodesic path} from $L$ to $L'$. Moreover, the geodesic path $\Gamma(L,L')$ 
is the {\em unique} path of minimal length between $L$ and
$L'$.
\end{prop}

\begin{proof}
The existential part of the lemma is easy to verify by using Smith bases of the
representatives $\la$ and $\la'$, and is left to the reader.
Note that the geodesic interval is symmetric. Indeed, letting 
$\Gamma(L',L)=(L'_0,L'_1,\ldots,L'_d)$, we have
$L_k=[M_{\la}+z^k\la]$ and $L'_{d-k}=[\la_{M}+z^{d-k}M]$.
By definition we have
\begin{eqnarray*}
\la_{M}+z^{d-k}M&=&z^{-v_M(\la)}\la+z^{-v_M(\la)-v_{\la}(M)-k}M \\
 &=&z^{-v_M(\la)-k}\left(z^k\la+z^{-v_{\la}(M)}M\right).
\end{eqnarray*}
Therefore $L'_{d-k}=L_k$.

Let us prove the uniqueness by induction on the distance $d=d(L,L')$.
For convenience, let any path \mbox{$([\la]=L_0,L_1,\ldots,L_{d-1},[M]=L_d)$}
be represented by its {\em normalised} sequence
$(\la,M_1,\ldots,M_{d-1},M_{d})$ of lattices $M_i\in L_i$ such that $v_{\la}(M_i)=0$.
We will first prove the following result: if $\Gamma'=(\la=L_0,L_1,\ldots,L_{d-1},M=L_d)$ 
is a path of minimal length, then the normalised sequence of lattices satisfies
$\la\supset M_1\supset \cdots \supset M_d$.
For $d=1$, this is the very definition of adjacency in $\B_n$.
Assuming that the claim is established for any pair of lattices at distance $\lsl d-1$, we have
$M_d\subset\la\supset M_1\supset \cdots \supset M_{d-1}$ for the normalised sequence of $\Gamma'$.
Since $M_{d-1}$ and $M_d$ are adjacent, there exists a unique $k\in \Z$ such that 
$z^k M_{d-1}\supset M_d\supset z^{k+1}M_{d-1}$. We know that $d(\la,M_{d-1})=d-1$, hence we have
$\la\supset M_{d-1}\supset z^{d-1}\la$, therefore we get 
$$z^k\la\supset z^kM_{d-1}\supset M_d\supset z^{k+1}M_{d-1}\supset z^{d+k}\la.$$
If $k>0$, then $v_{\la}(M)\gsl k >0$, which was excluded by assumption. But if $k<0$, then
$d(\la,M)<d$, which is also excluded. Thus we have $k=0$, and the claim is proved.

Now we turn to the proof of the uniqueness of the geodesic.
Since the claim is obvious for $d=1$, let us suppose that there exists a unique geodesic 
between any pair of vertices in $\B_n$ distant at most of $d-1$. 
Suppose then that $d(\la,M)=d$.
Let $\la=L_0\supset L_1\supset \cdots\supset L_{d-1}\supset M$ represent 
a path of minimal length, and $\la=M_0\supset M_1\supset \cdots\supset M_{d-1}\supset M$ the geodesic path from $\la$ to $M$.
By assumption, $d(\la,L_{d-1})=d-1$, therefore we have $\la\supset L_{d-1}\supset z^{d-1}\la$, 
and by definition, we have $M_{d-1}=M+z^{d-1}\la$. Therefore, we get 
\begin{eqnarray*}
M_{d-1}\cap L_{d-1}&=&(M+z^{d-1}\la)\cap L_{d-1}\\
                   &=&M+(z^{d-1}\la \cap L_{d-1})\mbox{ because }M\subset L_{d-1}\\
                   &=&M+z^{d-1}\la\mbox{ since }z^{d-1}\la\subset L_{d-1}\\
                   &=&M_{d-1}
\end{eqnarray*}
Thus $M_{d-1}\subset L_{d-1}$ holds. On the other hand, $M_{d-1}$ is the largest lattice
containing $M$, contained in $\la$ and adjacent to $M$. Since $L_{d-1}$ also satisfies
these conditions, we finally get $L_{d-1}=M_{d-1}$. By the induction assumption, the 
two geodesics coincide all along.
\end{proof}

If $\k=(k_1=0,\ldots,k_n)$ represent the sequence of \del\! of $M$ in $\la$, then the
elements $M_k$ of the (normalised sequence of the) geodesic path from $\la$ to $M$ have 
as \del\! in $\la$ the sequence $\k_k=(\min(k_i,k))$. The differences $T_k=\k_{k+1}-\k_k$
form what we will call the {\em elementary splitting $\k=T_1+\cdots+T_d$ of $\k$}. We have then
$$
\xymatrix{
\la\ar[r]^-{z^{T_1}}&M_1\ar[r]^-{z^{T_2}}&M_2\cdots M_{d-1}\ar[r]^-{z^{T_{d}}}&M
}
$$

\paragraph{Apartments.}
Let $\Phi=\{d_1,\ldots,d_n\}$ be an unordered set of one dimensional
$K$-vector subspaces of $V$ such that $d_1+\cdots +d_n =V$ ($\Phi$ is
called a {\em frame}). The set
$$[\Phi]=\{\la=\ell_1+\cdots+\ell_n \tq \ell_i\textrm{ is a lattice in }d_i\}$$
of lattices spanned over multiples of the vectors in $\Phi$ induces a
simplicial subcomplex in the affine building $\B_n$ called the {\em apartment}
spanned by $\Phi$.
For any lattice $\la\in \ra$, a {\em $\la$-basis} 
of the apartment $[\Phi]$ is a collection $(u_1,\ldots,u_n)$ of vectors such that
$u_i$ spans $d_i$ and
$v_{\la}(u_i)=0$. Such a family is unique up to permutation and to
multiplication of each $u_i$ by a scalar $\lb_i\in \h^*$. The
lattice is an element of the apartment $[\Phi]$ if and only if the
family $(u)=(u_1,\ldots,u_n)$ is actually a basis of the lattice $\la$. 
Equivalently, and without reference to a basis,
this means that $$\la=\bigoplus_{i=1}^n \la \cap d_i.$$ In
the general case, the lattice $\la_{\Phi}=\bigoplus_{i=1}^n \la \cap
d_i$ is the {\em largest sublattice} of $\la$ in the apartment $[\Phi]$.
The homothety class $L_{\Phi}=[\la_{\Phi}]$ is therefore the closest
point projection of $L=[\la]$ on $[\Phi]$.
Finally note that if $L,L'\in[\Phi]$, then $[\Phi]$ contains the whole geodesic path
$\Gamma(L,L')$.

\paragraph{Quotients.}
For vertices $L,L'\in \B_n$, we define the quotient $L'/L$ as the quotient 
module $\la'/\la$, where $\la\in L$ and $\la'\in L'$ satisfy $v_{\la}(\la')=0$.
We will sometimes say, for shortness' sake, that $\la,\la'$ are {\em $L$-normalised 
representatives} of $L,L'$. The quotient $L'/L$ is a 
well defined finite-dimensional $\C$-vector space. For any lattices satisfying
$\la'\subset N\subset \la$, let 
\begin{equation}\label{q}
\Psi_{\la,\la'}(N)=\left((N+\la')\cap\la\right)/\la'. 
\end{equation}

Let $E$ be the set of linear subspaces of $L'/L$, and 
$$[L',L]=\{M\in \B_n \tq \exists \la\in L,\la'\in L',N\in M,L\mbox{-normalised, 
such that } \la'\subset N\subset \la\}.$$
Formula (\ref{q}) defines a mapping $\appto{\Psi_{L,L'}}{\B_n}{E}$
which induces a poset isomorphism from $[L',L]$ to $F$, where
$F=\{ G\in E \tq \id G\subset G\}\subset E$. We will consider
the space $L'/L$ as a sub-$\C$-vector space of $\ov{L}^d=\la/\id^d\la$,
where $d=d(L,L')$. This definition is independent of the choice of $\la$ inasmuch
as there is a canonical isomorphism between $\la/\id^d\la$ and $\id^k\la/\id^{d+k}\la$.
In the special case $d=1$, the quotient space denoted with $\ov{L}=\la/\id\la$ is a $\C$-vector space of dimension $n$. 
We write $\Psi_L$ for  the isomorphism of simplicial complexes defined by relation (\ref{q}) between the set of neighbours
$\mathrm{lk}(L)$ of $L=[\la]$ in $\B_n$ (called the {\em link} of $L$) 
and the set $E$ of chains of linear subspaces of $\ov{L}$.

\subsection{Relative Flag of a Lattice}
\label{RelFlag}

Any lattice $M$ induces a natural flag in $\ov{\la}=\la/\id\la$ in the
following
way. For any $k\in\Z$, let $$M_k=(\id^{-k}M\cap \la)+\id\la \in
[\id\la,\la].$$
Let $(e)$ be a basis of elementary divisors of $\la$ for
$M$, and $I=\{1\lsl i \lsl n \tq k_i\lsl k\}$. Then $M_k$ admits $(u)$ as
basis where $u_i=e_i$ if $i\in I$ and $u_i=ze_i$ if $i\notin I$.
The spaces $M_k$ are thus embedded lattices, all belonging to
the interval $[\id\la,\la]$, so they take at most $n+1$ different
values. Their images $\ov{M}_k$ in the quotient space $\ov{\la}$ form a flag
$F_{\la}(M)$, and it is clear that $\ov{M}_{k-1}\subsetneq \ov{M}_{k}$
if and only if
$k$ is an elementary divisor of $M$. Let therefore $\kk_1,\ldots,\kk_s$ be
the distinct elementary divisors of $M$ in $\la$, with multiplicities
$n_i$. The subset of
indices corresponding to $\kk_j$ can
be written as $$I_j=[n_1+\cdots+n_j,n_1+\cdots+n_{j+1}-1].$$ The
lattices $M_k$ and $M_{\ell}$ coincide if and only if there exists $i$
such that $\kk_i\lsl k,\ell
<\kk_{i+1}$ (with the conventions $\kk_0=-\infty$ and $\kk_{s+1}=+\infty$).
Therefore the flag $F_{\la}(M)$ has exactly length $s$, and its signature is
equal to the sequence $(n_1,\ldots,n_s)$.
Its components can be indexed either as $\ov{M}_{\kk_i}$, by the value of the
elementary divisor $\kk_i$ it is attached to (if known), or
as $\ov{M}_i$ by its index in the flag (here $i$). In this latter
case, we will also
use the notation $F_{\la}^i(M)$. It will hopefully be always
clear what convention we are using.

Note that the flag $F_{\la}(M)$ corresponds under the isomorphism
$\Psi_{[\la]}^{-1}$ to a canonical simplex in $\B_n$ containing the 
vertex $L=[\la]$. Modulo homothety, one can define the flag 
$F_L(L')$ in the space $\ov{L}$ defined in section~\ref{build},
and the following result holds.

\begin{lemm}
\label{geo:flag}
Let $L,L'\in \B_n$ be vertices in $\B_n$, let $d=d(L,L')$ and let $(L_0,\ldots,L_d)$ be 
the geodesic interval $\Gamma(L,L')$. For $0\lsl k \lsl d$, the flag $F_L(L_k)$ is given by
$$F_L(L_k):F_{\la}^0(M)\subset\cdots\subset F_{\la}^{\nu}(M)\subset \ov{L},$$ where $\nu$ is the index such that
$k+v_{\la}(M)\in I_{\nu}$, for any representatives $\la\in L$ and $M\in L'$.
Moreover, the flags $F_{L_k}(L)$ and $F_{L_k}(L')$ in $\ov{L_k}$ have
supplementary first components if $k$ is a normalised elementary
divisor of $L'$ in $L$.
\end{lemm}

\begin{proof}
Take representatives $\la,M$ of $L,L'$, and a Smith basis $(e)$ of $\la$ for $M$. Let
$\k_{\la}(M)=(k_1,\ldots,k_n)=(\kk_1I_{n_1},\ldots,\kk_sI_{n_s})$, and assume that $\kk_1=0$.
Then suitable representatives of $L_k$ are the lattices $M^k=M+z^k\la$, which
admit as bases $(e^k)=(z^{\min(k_i,k)}e_i)_{i=1,\ldots,n}$. Therefore
the \del\! of $\la$ and $M$ in $M^k$ are respectively 
$$K_1=(\max(k-k_i,0))\mbox{ and }K_2=(\max(0,k_i-k)).$$
Let $j$ be the index such that $\kk_{j}\lsl k<\kk_{j+1}$. 

We must here distinguish two cases. If $\kk_{j}<k<\kk_{j+1}$, then we have
$$K_1=(kI_{n_1},(k-\kk_2)I_{n_2}\ldots,(k-\kk_j)I_{n_{j}},0_{n_{j+1}},\ldots,0_{n_s})$$ and
$$K_2=(0_{n_1},0_{n_2},\ldots,0_{n_{j}},(\kk_{j+1}-k)I_{n_{j+1}},\ldots,(\kk_{s}-k)I_{n_{s}}).$$
Then obviously the induced flags $\F=F_{L_k}(L)$ and $\F'=F_{L_k}(L')$ have
respective signatures $(n_{j+1}+\cdots+n_s,n_j,\ldots,n_1)$ and 
$(n_1+\cdots+n_j,n_{j+1},\ldots,n_s)$. Their first components $\F_1$
and $\F_1'$ are supplementary subspaces of $L_k/\id L_k$. If $k=\kk_{j}$ however, 
we have $\F_1\cap \F_1'=<z^{\kk_j}\ov{e}_{\nu_j},\ldots, z^{\kk_j}\ov{e}_{\nu_{j+1}-1}>$.
\end{proof}

\subsection{Forms}
\label{constant}

Fix a lattice $\la$ and let for simplicity
$\map{\Phi_{\la}}{\ra}{W(\la/\id\la)}{M}{(\F_{\la}(M),\k_{\la}(M))}$.
This map is clearly surjective, but as clearly not injective. The
objective of this section is to show how to invert it.

Let a {\em form} in $\la$ be a $\C$-vector subspace $Y$ of $\la$
spanned by an $\h$-basis $(e)$ of $\la$. If we fix a form $Y$ in $\la$
(that is a $\C$-linear section of the canonical projection
$\appto{\pi}{\la}{\la/\id\la}$)
then there is a unique way to lift the quotient module
$\la/\id\la$ in $Y$, that is, there is a well-defined isomorphism
$$\isomto{\varphi_Y}{Y}{\la/\id\la}.$$
Let $\B(Y)$ be the sub-building of $\B_n$ composed of apartments $[\Phi]$
which are spanned by a basis of $Y$. For a given flag $F\in \fl(\la/\id\la)$, 
let us define the {\em $Y$-fiber of $F$} as $$\Psi_{Y}(F)=\{M\in \B(Y) \tq F_{\la}(M)=F\}$$ 
We will say that $Y$ is a {\em Smith form} for $M$ if $M\in\Psi_{Y}(F)$.

\begin{lemm}\label{flag:lattice}
Let $\la$ be a lattice in $V$ and $Y$ be a form in $\la$.
\begin{enumerate}[i)]
\item For any admissible pair $(F,\k)$ of $\la/\id\la$, there exists a 
unique lattice $M=\Psi_Y(F,\k)$ in $\Psi_{Y}(F)$ such that $\k_{\la}(M)=\k$.
\item For any basis $(e)$ of the lattice $\la$, there
exists a {\em unique} $\C$-basis $(e_Y)$ of the form $Y$ whose image
in $\la/\id\la$ coincides with the image of $(e)$. 
\end{enumerate}
We call $(e_Y)$ the {\em $Y$-basis of $(e)$}.
\end{lemm}

\begin{proof}
For any $\C$-basis $(e)$ of $Y$ which respects the flag $F$, put
$M=\bigoplus_{i=1}^n z^{k_i}e_i$. Let $(\te)$ be another basis of $Y$
and $\tM=\bigoplus_{i=1}^n z^{k_i}\te_i$. The matrix of the change of
basis from $(z^{\k}e)$ to $(z^{\k}\te)$ is equal to
$P=z^{\k}Cz^{-\k}$, where $C\in \Gc$ is the matrix of the change of
basis from $(e)$ to $(\te)$. By definition of the parabolic subgroup
$P_F$, one has $z^{\k}Cz^{-\k}\in \Go\iff C\in P_F$, hence $M=\tM$ if
and only if $(e)$ and $(\te)$ both respect the flag $F$.
The second claim is straightforward. Note that the gauge from the 
basis $(e)$ to its $Y$-basis is always of the form $P=I+zU\in \Go$.
\end{proof}

The correspondence $\Psi_Y$ is therefore a bijection between the set
$W(Y)$ of admissible pairs of $Y$ and the sub-building $\B(Y)$. 
Let $F_Y(M)=\varphi^{-1}_Y(F_{\la}(M))$ be the lifting of the relative flag of
$M$ in $Y$, and define
$$\F_Y(M)=\tens{F_Y(M)}{\h}{K}.$$
The signatures of these three flags are all equal to the
multiplicities $(n_1,\ldots,n_s)$ of the original lattice $M$. Putting
all this together, we have the following definition.

\begin{defn}
Let $\la$ be a lattice in $V$, and $Y$ a form in $\la$. Let 
$M$ be a lattice in $V$, let $F=F_{\la}(M)$ be the induced flag and 
$\k=K_{\la}(M)$ its elementary divisors.
\begin{enumerate}[i)]
\item The flag $\F_Y(M)$ of $K$-vector spaces in $V$ is called the {\em $Y$-flag of $M$}.
\item The lattice $M_Y=\Psi_Y(F,\k)\in \Psi_{Y}(F)$ is called the $Y$-lattice of $M$.
\end{enumerate}
\end{defn}

For any two forms $Y$ and $\ti{Y}$, the set of gauges between bases
of $Y$and $\ti{Y}$ is an element of the double coset $\Gc\backslash\Go/\Gc$. 
Let $z$ be a uniformising parameter. With the 
convention that $\deg_z P= \infty$ if $P\in \Go\backslash\gl(\C[z])$, the
following definition makes sense.

\begin{defn}\label{form:dist}
Let $Y,\ti{Y}$ two forms in $\la$. The {\em $z$-distance}
$\delta_z(Y,\ti{Y})$ is defined as $\min(\deg_z P,\deg_z P^{-1})\in \N\cup\{\infty\}$ 
for any gauge $P$ from a basis of $Y$ to a basis of $\ti{Y}$.
\end{defn}

\begin{lemm}
\label{poly:form}
If $d=d(\la,M)$, then for any form $Y$ of $\la$, and any uniformising
parameter $z$, there exists a Smith form
$\ti{Y}$ for $M$ at a $z$-distance $\delta_z(Y,\ti{Y})\lsl d-1$.
\end{lemm}

\begin{proof}
There exists a Smith form $Y'$ of $\la$ for $M$. Let $P=P_0+P_1z+\cdots\in \Go$ 
be a gauge corresponding to a basis change from $Y$ to $Y'$. Let
$\ov{P}=P_0+\cdots+P_tz^t$, and let $\ti{Y}$ be the form obtained 
by this gauge transformation, as explained in the following scheme.
$$
\xymatrix{
\tilde{Y} \ar[d]^-{z^{\k}}& Y \ar[r]^-{P} \ar[l]_-{\ov{P}} & Y' \ar[d]^-{z^{\k}}  \\
M & & M \ar[ll]^-{Q}
}
$$
We have $Q=z^{-\k}P^{-1}\ov{P}z^{\k}=(\ti{P}_{ij}z^{k_j-k_i})$
where $\ti{P}=P^{-1}\ov{P}$. By construction,
we have $\ti{P}=P^{-1}(P-(P-\ov{P}))=I+z^{t+1}U$ with $U\in \gl(\h)$.
As soon as $t\gsl d-1$, we have $Q\in \Go$, hence
the form $\ti{Y}$ is a Smith form for $M$.
\end{proof}

The definitions of distance and index in the Bruhat-Tits building suggest the following.

\begin{defn}\label{VB:dist}
Let $\E,\F\in H$ be two holomorphic vector bundles. The {\em distance}
between $\E$ and $\F$ and the {\em index} of $\F$ with respect to $\E$
are defined as the integers
$$d(\E,\F)=\max_{x\in X}d(\E_x,\F_x)\mbox{ and }[\E:\F]=\sum_{x\in X}[\E_x:\F_x]$$
where quantities on the right-hand side denote those defined in the local
Bruhat-Tits building at $x$ by relation~(\ref{d-i}).
\end{defn}

\section{Birkhoff-Grothendieck Trivialisations}
\label{BG}

The central result in the theory of holomorphic vector bundles on $X=\P$ is
the Birkhoff-Grothendieck theorem, which states that any such bundle
is isomorphic to a
direct sum of line bundles. In this section, we investigate what
properties of the vector bundle
can be retrieved by considering only the Bruhat-Tits building at a
point $x\in X$. In what follows, we take $X=\P$.

\subsection{The Birkhoff-Grothendieck Property}
\label{BG:triv}

According to section \ref{HVB}, a holomorphic vector bundle $\E\in H$ is
completely described by the coset $\ra=[\E]\in H/R_x$ and the lattice
$\la=\E_x\in\ra$.
Let us take up the notations of section \ref{sdc} again.
Let $V$ denote the meromorphic stalk $\V_x$ and let $\B$ be the corresponding
Bruhat-Tits building. Let $\B_0$ the subset of {\em trivialising
lattices} of $\B$. Strictly speaking, these are the lattices $M\in\ra$ such that the
extension $\E^{M}$ gives a
{\em quasi-trivial} vector bundle, but we will not bother much to make
the difference, since we will get a trivial bundle by simply tensoring
by a line bundle. The space $\Gamma(X,\F)$ of sections of a trivial
bundle $\F=\E^M$ induces, by taking stalks at $x$, a form $Y_M$
in the corresponding lattice $M=\F_x$, that we call the {\em global
form} of $M$.

It follows from the Birkhoff-Grothendieck theorem that the set $\B_0$
is always non-empty. We do not know if this is actually a weaker result.
However, if we admit this possibly weaker result, we can deduce from it
an elementary algebraic proof of the Birkhoff-Grothendieck theorem
that displays
quite nicely the geometric properties of the local Bruhat-Tits building.
First we start by making the link between the Birkhoff-Grothendieck
theorem and
the algebraic structure of the local lattices. Let us say that $\E$ has the
Birkhoff-Grothendieck property if $\E\simeq\bigoplus_{i=1}^n \L_i$
where $\L_i$ are holomorphic line bundles. Then the following characterisation is
straightforward.

\begin{lemm}
\label{BG:del}
A vector bundle $\E\in H$ has the Birkhoff-Grothendieck property
if and only if there exists a trivialising lattice $M\in \B_0$
and a Smith basis $(e)$ of $M$ with respect to $\la=\E_x$
that is simultaneously a $\C$-basis of the global form $Y_M$ of $M$.
\end{lemm}

The previous result can be understood in the following sense: if we
put
$$\E\simeq \bigoplus_{i=1}^n \h(a_i)\mbox{ with }a_1\gsl\cdots \gsl
a_n,$$
then there is a basis $(\vep)$ of $\E_x$ such that the matrix of the
change of basis to $M=\F_x$ is given by the diagonal matrix
$T=\diag(a_1,\ldots,a_n)$ of elementary divisors, where $a_i\gsl
a_{i+1}$. We sum this situation by the diagram 
$\appname{\E_x}{\F_x}{z^T}$, where $z$ is a local
coordinate at $x$. Note that Bolibrukh uses the inverse convention
with types $\h(-c_i)$ and $c_1\lsl\cdots\lsl c_n$.

In this case, we say that $\F$ is a {\em Birkhoff-Grothendieck
trivialisation of $\E$ at $x$}, the basis $(e)$ a
{\em Birkhoff-Grothendieck basis} of $\F$ for $\E$, and the apartment
$[\Phi]$ spanned by $(e)$, a {\em Birkhoff-Grothendieck apartment}
for $\E$. To avoid multiplying definitions, we will say that a basis
$(\vep)$ of $\V_x$ is {\em Birkhoff-Grothendieck} if there is a local
coordinate $z$ at $x$ and a diagonal integer matrix
$T=\diag(a_1,\ldots,a_n)$ such that $z^T(\vep)$ is a basis of a
trivialisation of the coset $\ra\in H/R_x$.

\begin{rem}
When $X=\P$, a line bundle $\L$ is characterised by its degree.
Recall that if the integers $a_i=\deg \L_i$ satisfy $a_1\gsl \cdots
\gsl a_n$, then the sequence $T(\E)=(a_1,\ldots,a_n)$ is unique and 
called the {\em type} of $\E$. The group of monopole gauges is described
by the group of unimodular polynomial matrices $\G(\C[T])$, that is matrices
of the form $$P=P_0+P_1T+\cdots+P_kT^k\mbox{ where }\exists \alpha\in \C -\{0\},
\det P=\alpha\mbox{ for all }T.$$
\end{rem}

We state now the following result separately for further reference.

\begin{prop}\label{del=type}
Assume $X=\P$, and let $\E\in H$ be a holomorphic vector bundle.
The type of the bundle $\E$ is equal to the sequence of elementary
divisors $\k_{\E_x}(\F_x)$ (in reverse order) of the stalk $\E_x$ 
with respect to $\F_x$ (viewed as lattices in $\V_x$), 
for any Birkhoff-Grothendieck trivialisation $\F$ of $\E$ at any $x\in X$.
\end{prop}

Let $\E$ have type $T(\E)=(a_1,\ldots,a_n)$. The {\em triviality index} 
\mbox{$i(\E)=\sum_{i=1}^n(a_1-a_i)$} measures how far $\E$ is from being quasi-trivial.
In a more ``intrinsic'' way, we can define it as the sum of the indices of the dual bundles
$i(\E)=[\E^*:\F^*]$ for any Birkhoff-Grothendieck trivialisation $\F$ of $\E$.

The Birkhoff-Grothendieck trivialisations of a bundle $\E$ are as a rule not unique.
\begin{lemm}
\label{BGTriv:all}
Let $\la\in \ra\in H/R_x$ represent a bundle of type $\k=(k_1,\ldots,k_n)$. Then the set of Birkhoff-
Grothendieck bases of $\la$ is the orbit of any Birkhoff-Grothendieck basis $(e)$ under the
{\em $\k$-staged parabolic group} $G_{\k}=\{P\in\Go \tq \deg(P_{ij}) \lsl k_i-k_j \}$.
\end{lemm}

\begin{proof}
Consider two Birkhoff-Grothendieck trivialisations $M,\tM$
of $\la$, like in the following diagram:
$$
\xymatrix{
\la \ar[d]_-{P} \ar[r]^-{z^{\k}} & Y_M \ar[d]^-{\Pi} \\
\la \ar[r]^-{z^{\k}} & Y_{\tM}
}
$$
Since $v(P_{ij})\gsl 0$, the gauge $\Pi=z^{-\k}Pz^{\k}$ is a monopole gauge 
if and only if $\deg(P_{ij}) \lsl k_i-k_j$. 
\end{proof}

Since any $\k$-staged parabolic group in dimension 1 is equal to $\C^*$, a {\em line bundle} $\L$ 
does however admit a unique trivialisation $T_x(\L)$ at $x$.

\begin{rem}
\label{K-para}
If the type is ordered by decreasing values, then the matrix $P$ is in block-upper-triangular
form with respect to the blocks of equal elements of $\k$ (that is,
{\em $\k$-parabolic}), and so is $\Pi$.
\end{rem}

\subsubsection{Transporting a Birkhoff-Grothendieck Trivialisation}
\label{trans:BG}

We explain here what we understand by transporting the trivialisation point 
from $x$ to $y\in X$. According to the previous section, a Birkhoff-Grothendieck
trivialisation of a bundle $\E$ at $x$ is a trivial bundle $\F$ such that
$(\E,\F)\in R_x$, and such that there exists a basis $(\sg)=(\sg_1,\ldots,\sg_n)$ of global sections 
of $\F$ and an integer sequence $\k=(k_1,\ldots,k_n)$, such that 
$(e)=(t^{-k_1}\sg_1,\ldots,t^{-k_n}\sg_n)$ spans the stalk $\E_x$ over the 
local ring $\h=(\ox)_x$, where $t$ is a local coordinate at $x$. This coordinate $t$ can 
be arbitrarily chosen, since the local behaviour of $\E$ only depends on the local ring $\h$.
However, if we choose as coordinate $t$ a meromorphic function on $X$, then the 
sections $(e)$ form a basis of global (meromorphic) sections of $\V$. The $\ox$-module $\ti{\F}$
spanned by $(e)$ in this case does coincide with $\E$ at $x$, and differs from it at most
on the support of the divisor of the function $t$. When $X=\P$, we can obviously find a 
function $t$ with divisor $(t)=x-y$ for any arbitrary point $y\neq x$. In this case,
the bundle $\ti{\F}$ is a Birkhoff-Grothendieck trivialisation of $\E$ at $y$. It is 
clearly independent of the global basis $(\sg)$ of $\F$, which is defined up to a $(-K)$-parabolic
constant matrix $C\in \Gc$, and of the specific meromorphic function $t$, which is only 
defined up to a non-zero constant. We call $t_y(\F)=\ti{\F}$ the {\em transport at $y$} of the
Birkhoff-Grothendieck trivialisation $\F$ of $\E$ at $x$.

Understood otherwise, this is the description of a non-trivial bundle $\E$ by means 
of {\em two trivial} bundles $\F$ and $\ti{\F}$ coinciding outside $\{x,y\}$, and glued
along the cocycle $g=t^{\k}$, where $(t)=x-y$.

\subsubsection{The Harder-Narasimhan Flag}
\label{HN:eldiv}

The Harder-Narasimhan filtration $\mathrm{HN}(\E)$ of $\E$ over $\P$ can be
obtained easily (see \cite{Sa} p. 65) from a decomposition
$\E=\bigoplus_{i=1}^n \L_i$ of $\E$, as a direct sum of line bundles $\L_i\simeq \h(a_i)$
of the appropriate degree, by $$F^k(\E)=\bigoplus_{i \tq
a_i \gsl k}\L_i.$$ Note that such a direct sum  $\L=(\L_1,\ldots,\L_n)$ induces at $x$ a
canonical Birkhoff-Grothendieck trivialisation $\L_x(\E)=\bigoplus_{i=1}^n T_x(\L_i)$.
Locally, the Harder-Narasimhan filtration can be defined as follows. Let $(e)$ be a
Birkhoff-Grothendieck basis of $\E_x$. The {\em Harder-Narasimhan flag
$\hn_{\la}$ of $\V_x$} is defined by 
\begin{equation}\label{HN:def}
F^k=\bigoplus_{i \tq a_i \gsl k}K e_i
\end{equation}

\begin{lemm}
\label{localHNF}
Let $\E$ be a holomorphic vector bundle over $X=\P$. For $x\in
X$, let $V=\V_x$ and $\la=\E_x$. 
Then the Harder-Narasimhan flag $\hn_{\la}$ of $V$ defined by relation~(\ref{HN:def})
is independent of the Birkhoff-Grothendieck trivialisation $\F$ 
and basis $(e)$ appearing in the definition. Moreover, let $\pi^{\E}_x$ be the projection
$\app{\E}{E=\E_x/\id_x\E_x}$. Then the following hold:
\begin{enumerate}[i)]
\item The $\h$-flag $\hn_{\la}\cap\la$ coincides with the stalk $\hn(\E)_x$ of the
Harder-Narasimhan filtration of $\E$.\label{localHNFi}
\item For any Birkhoff-Grothendieck trivialisation $\F$ of $\E$ at $x$,
letting $M=\F_x$, the flag $\pi^{\F}_x(\hn_{\la}\cap M)$ coincides with 
the relative flag $F_{M}(\la)$ defined in section~\ref{RelFlag}\label{localHNFii}.
\item Conversely, for any flag $F'$ which is transversal to the flag $\pi_x^{\E}(\hn(\E)_x)$ 
in $E=\la/\id_x\la$, there exists a Birkhoff-Grothendieck trivialisation $\F'$ of $\E$ at $x$,
such that the relative flag $F_{\la}(\F'_x)$ in $E$ coincides with $F'$\label{localHNFiii}.
\end{enumerate}
\end{lemm}

\begin{proof}
The two first assertions are straightforward enough. Let us prove the
third one. Let $T=\diag(a_1I_{n_1},\ldots,a_sI_{n_s})$ with
$a_i>a_{i+1}$ be the matrix of elementary divisors corresponding to
the transformation $\appname{\la=\E_x}{M=\F_x}{z^T}$ in
the basis $(e)$, and let $(\vep)$ be the basis $z^T(e)$. 
Let for simplicity of notation $\nu_i=\sum_{1\lsl k \lsl i} n_i$.
The $(n-i+1)$-th component $F_i$ of the
flag $F_{\la}(M)$ induced by $M$ in $E=\la/\id_x\la$ is spanned by 
$(\ov{e}_{\nu_i+1},\ldots,\ov{e}_n)$, where $\ov{e}_k$ is the 
image of the basis vector $e_k$ in $E$, whereas, according to what has
just been established, the image of the Harder-Narasimhan flag has its
$i$-th component spanned by $(\ov{e}_1,\cdots,\ov{e}_{\nu_i})$, hence both
flags are transversal to each other. Any other 
Birkhoff-Grothendieck trivialisation $\tM$ is obtained from $(\vep)$ by a 
monopole gauge transform $\Pi$ such that $P=z^T\Pi z^{-T}\in \Go$. According to
Note~\ref{K-para}, $\Pi$ is block-upper-triangular with respect to the blocks of 
equal elements of $T$, hence so is $P$. For any such $P\in \Gc$,
the matrix $z^{-T}Pz^T$ is a monopole. The orbit of $(\vep)$ under the
set of the constant $T$-parabolic matrices covers the set of all flags in $E$
which are transversal to the image of $\hn(\E)_x$ in $E$.
\end{proof}

For any Birkhoff-Grothendieck trivialisation $\F$ of $\E$ at $x$,
let $Y=\Gamma(X,\F)$ be the $\C$-vector space of global sections of $\F$. 
The Harder-Narasimhan filtration $\hn(\E)$ also induces a canonical
filtration $\hn_{\la}(Y)$ of $\C$-vector spaces of $Y$. To avoid defining
new concepts, we will also refer to this filtration as the 
{\em Harder-Narasimhan filtration of $Y$}. Note that it is depends completely
on the lattice $\la\in \ra$.

\subsection{Modification of the type}
\label{type:modif}

We wish to answer algebraically the following question: ``What does
the type of $\E$ become when the stalk $\E_x=\la$ at $x$ is replaced
by another
lattice $\ti{\la}$?'' It turns out that the question can be very
explicitly answered when
the lattice $\ti{\la}$ is not too far from $\la$, namely at distance 1 in the
graph-theoretic distance of the Bruhat-Tits building. The following
proposition generalizes a result of Gabber and Sabbah.

\begin{prop}
\label{GS}
Let $\E\simeq\bigoplus_{i=1}^n \h(a_i)$ be a holomorphic vector bundle
on $X=\P$,
with $a_1\gsl\cdots\gsl a_n$, and let $x\in X$. Let $\ti{\la}\in
\ra_x$ be a lattice such that
$\id_x\E_x\subset\ti{\la}\subset \E_x$. Let $E=\E_x/\id_x\E_x$ be the
local fiber at $x$, let
$F:F_0=0\subset F_1 \subset \cdots \subset F_s=E$ be the flag induced
in $E$ by the
Harder-Narasimhan filtration of $\E$, and $W=\ti{\la}/\id_x\E_x$ be
the image of $\ti{\la}$.
Assume that the type of $\E$ is written as
$$a=(\underbrace{a_1,\ldots,a_1}_{n_1\textrm{ times}},
\underbrace{a_2,\ldots,a_2}_{n_2\textrm{ times}},\ldots,
\underbrace{a_s,\ldots,a_s}_{n_s\textrm{ times}}).$$
Then the modified bundle $\F=\E^{\ti{\la}}$ has type
$$\ti{a}=(\underbrace{a_1,\ldots,a_1}_{m_1\textrm{ times}},
\underbrace{a_1-1,\ldots,a_1-1}_{n_1-m_1\textrm{ times}},\ldots,
\underbrace{a_s,\ldots,a_s}_{m_s\textrm{ times}},
\underbrace{a_s-1,\ldots,a_s-1}_{n_s-m_s\textrm{ times}})$$ where
$m_i=\dim_{\C} F_i \cap W-\dim_{\C} F_{i-1} \cap W$.
\end{prop}

\begin{proof}
This is explained in the following scheme. Let $\la=\E_x$, and let $t$
be a local
coordinate at $x$. Let $K=\diag(a_1,\ldots,a_n)$ be the elementary divisors
of the Birkhoff-Grothendieck trivialisation $M$ in $\la$ (or, in this case, 
the type of $\E$).
$$
\xymatrix{
Y_M \ar[r]^-{t^{-K}P_0t^{K}}& Y_{\tilde{M}} \\
\la:(e) \ar[u]^-{t^{K}} \ar[r]^-{P_0} \ar[d]^-{\pi} & (\vep)
\ar[u]^{t^{K}} \ar[r]^-{t^{T}} & (\tilde{e}):\tilde{\la}\ar[ul]_-{t^{K-T}}\\
E=\la/t\la:(\ov{e}) \ar[r]^-{P_0} & (u) : W \ar[ur]_-{\Phi_{\la}^{-1}}
}
$$

Let $(e)$ be a basis of $\la$, such that $(\sg)=(t^{K}e)$ is a
basis of the form $Y_M$. Under the canonical projection
$\appto{\pi}{\la}{E=\la/t\la}$, the $\mathrm{HN}$ filtration of $\la$
descends to a flag of $\C$-vector spaces $F:0=F_0\subset \cdots
\subset F_s=E$, and the quotient basis $(\ov{e})$ is a basis
respecting this flag. Let $t\la\subset \ti{\la} \subset \la$ be the
new lattice, and let $W\subset E$ be the subspace it is projected upon
by $\pi$. Let $(u)$ be a basis respecting both $W$ and the flag $F$, and
let $P_0$ be a change of basis from $(e)$ to $(u)$.
Consequently, the matrix $P_0$ belongs to the parabolic subgroup $P_F$
stabilising the flag $F$, therefore
it is block-upper-triangular, with blocks given by the equal elements
among the $a_i$. Define now
the basis $(\vep)$ of $\la$ as the image of $(e)$ under the constant
gauge $P_0$. Here is where $d(\la,\ti{\la})\lsl 1$ is important: the
basis $(\vep)$ is a Smith basis of $\ti{\la}$ (this would be not
necessarily true if the lattices were further apart). Let
$T=\diag(t_1,\ldots,t_n)$ be the diagonal matrix such that $t_i=0$ if
$\pi(\vep_i)\in W$ and $t_i=1$ otherwise. Then $(\ti{e})=t^T(\vep)$
is a basis of $\ti{\la}$. Let now $(\ti{\vep})=t^{K}(\vep)$ be the
basis of $\ti{M}$ deduced from $(\vep)$. The matrix of the basis
change from $\ti{\la}$ to $\ti{M}$ corresponding to the bases $(\sg)$
and $(\ti{\vep})$ is equal to $Q=t^{-K}P_0t^{K}=(P_0)_{ij}t^{k_j-k_i}$.
Now, since $P_0\in P_F$, we have $(P_0)_{ij}=0$ whenever $k_i-k_j<0$.
Therefore this gauge $Q=\frac{1}{t^k}Q_k+\cdots+Q_0$ is a Laurent
polynomial in $t$ with only non-positive terms, where moreover $Q_0\in
\Gc$. Since $X=\P$, it is possible to choose as local coordinate
at $\infty$ a meromorphic function with divisor $(\infty)-(0)$,
namely $t=1/z$. Accordingly, $Q$ is a polynomial in $z$,
whereas $\det Q=\det P_0\in \C^*$. Hence $Q\in \G(\C[z])$ is a
monopole gauge. Since $(\sg)$ was a basis of global meromorphic sections of
$E$, then $(\ti{\vep})$ also is. Therefore $\ti{M}\in B_0$ is a
trivialising lattice.
Moreover, $\ti{M}$ is a Birkhoff-Grothendieck trivialisation of both $\E$
and $\F=\E^{\ti{\la}}$, because the basis $(\ti{\vep})$ is a Smith basis for
$\la$ and $\ti{\la}$. Therefore, we can explicitly compute the new
elementary divisors
of $\ti{\la}$ in $\ti{M}$, which are given by the matrix $K-T$.
Summing up, we see that the change of lattice has subtracted 1 to all
the elementary
divisors corresponding to the vectors of the basis $(\vep)$ whose
image under $\pi$ do not fall into the subspace $W$. We obtain the Harder-
Narasimhan filtration of the modified bundle by reordering the type by
decreasing values.
\end{proof}
This generalises the construction given by Sabbah based on an
idea of O. Gabber in \cite{Sa}, prop. 4.11 (where only the case where $W$
is 1-dimensional is tackled). Based on this result, the Birkhoff-Grothendieck
theorem would get an immediate proof.

\begin{cor}[Birkhoff-Grothendieck theorem]
Any vector bundle $\E\in H$ over $\P$ has the Birkhoff-Grothendieck property.
\end{cor}

\begin{proof}
According to proposition~\ref{GS}, if a bundle $\E$ has the
Birkhoff-Grothendieck property,
so does $\E^M$ for any lattice $M\in [\E]_x$ which is adjacent to
$\E_x$. However, according
to lemma~\ref{geo}, two lattices are always connected by a path of
adjacent lattices.
Since a trivial bundle obviously has the Birkhoff-Grothendieck
property, the result is established.
\end{proof}

Note that an arbitrary trivialisation $M$ at $x$ of a vector bundle
$\E$ is not necessarily a Birkhoff-Grothendieck one. Another obvious but
useful remark is that, if $M$ is a Birkhoff-Grothendieck trivialisation of $\la$,
then it also is for any lattice $\la'$ on the geodesic path $\Gamma(\la,M)$.

Proposition~\ref{GS} allows to construct
effectively from an arbitrary trivialisation $M$ a Birkhoff-Grothendieck
one, by following geodesics in the Bruhat-Tits building from
$M$ to $\E_x$. The following result shows how to start the construction.

\begin{cor}
\label{BG:adj}
Let $M\in \ra$ be a trivialising lattice in $\B_n$. Then any adjacent
lattice $\la$ admits $M$ as Birkhoff-Grothendieck trivialisation.
More precisely, let $Y\subset M$ be the global form of $M$. For
any basis $(e)$ respecting $W=\la/\id M$, the $Y$-basis $(e_Y)$ is a Smith basis for
$\la$.
\end{cor}

\begin{proof}
Let $W=\la_M/\id M$ and let $T=\left(\begin{array}{cc}0_r & 0 \\0 &
I_{n-r}\end{array} \right)$ be the elementary divisors of $\la_M$ 
with respect to $M$.
Assume that $(e)$ satisfies the assumptions of
the corollary. Then, according to lemma~\ref{flag:lattice},
the $Y$-basis $(e_Y)$ is obtained by a gauge $P=I+tU\in \Go$.
Putting $U=\left(\begin{array}{cc}U_{11} & U_{12} \\U_{21} &
U_{22}\end{array}
\right)$, we have
$$
\xymatrix{
Y:(e_Y) \ar[r]^-{t^{T}} & \la_M \ar[r]^-{t^{v_M(\la)}I}& \la\\
M:(e)\ar[u]^-{I+tU} \ar[r]^-{t^{T}} & \la_M \ar[u]_-{\tilde{U}\in \Go} 
}
$$
since $\ti{U}=t^{-T}(I+tU)t^T=\left(\begin{array}{cc}I_r+tU_{11} &
t^2U_{12}\\U_{12} &  I_{n-r}+tU_{22}\end{array}
\right)$. The basis $(e_Y)$ is therefore indeed a Smith basis of $M$ for $\la$.
Since it is a basis of the global form of $M$, the result follows, and in particular, the
Harder-Narasimhan filtration of the corresponding bundle is equal to the
$Y$-lifting of the flag $(0\subset W \subset M/\id M)$.
\end{proof}

\subsubsection{An Algorithm to compute a Birkhoff-Grothendieck Trivialisation}
\label{BG:algo}

Let $x\in X$, and let $\ra=[\E]_x$ be the $R_x$-equivalence class of $\E$.
Let $\la=\E_x$ and $M=\F_x\in \ra$ where $\F$ is an arbitrary trivialisation of
$\E$ at $x$. In this local setting, we ``see'' the global sections of $\F$ as
the {\em global form} $Y\subset M$. According to lemma~\ref{geo}, 
putting $\la^k=\la_M+\id^kM$ for $k\in \Z$,
the sequence $(\la^0,\ldots,\la^d)$, where $d=d(\la,M)$, forms a chain
of adjacent lattices from $M$ to $\la_M$. By successive
applications of proposition~\ref{GS}, we construct a Birkhoff-Grothendieck trivialisation 
$M_k$ of $\la^{k}$. Let us explain this precisely.

If there existed a Smith basis of $M$ for $\la$ which spans simultaneously $Y_M$, the lattice $M$
would be a Birkhoff-Grothendieck trivialisation of $\la$, and the sequence $\k=-T$ would 
represent up to homothety the type of $\E$. This is generally not the case.

\begin{lemm}\label{geodesic:step}
Let $N\in\ra$ be a lattice admitting $M\in\ra^0$ as Birkhoff-Grothendieck trivialisation. If $v_N(\la)=0$, then
there exists a trivial lattice $\tM$ which is a Birkhoff-Grothendieck
trivialisation of both $N$ and $\la+\id N$.
\end{lemm}

\begin{proof}
By assumption, there exists a Birkhoff-Grothendieck basis $(e)$ of $M$ for $N$. Let $(y)$ be the
corresponding basis of $N$, and let $(\vep)$ be a Smith basis of $N$ for $\la$. We also assume that the 
elementary divisors $T$ of $N$ in $M$ and $T'$ of $\la$ in $N$ are ordered by increasing values.
The gauge $U$ from $(y)$ to $(\vep)$ can be factored as
\begin{equation}\label{tiu}
U=U_0(I+tU')\mbox{ with }U_0\in \Gc.
\end{equation}

According to lemma~\ref{localHNF}, the Harder-Narasimhan filtration of $N$ induces in $E=N/\id N$
a flag $F$ spanned by the basis $(\ov{y})$, whereas the flag $F'$ induced by $\la_M$ in $E$
is spanned by $(\ov{\vep})$. Let $B$ be the standard Borel subgroup of $\Gc$.
By the Bruhat decomposition, the group $\Gc$ is a disjoint union of double cosets $$\Gc=\coprod_{w\in W}BwB$$
where $W$ is the Weyl group $W=S_n$. The constant term $U_0$ of the gauge $U$ belongs to only one such cell: 
let $w\in S_n$ be the label of the corresponding Schubert cell. We have a decomposition $U_0=QP_wQ'^{-1}$ with $Q,Q'\in B$, where
$P_w$ is the matrix representation of the permutation $w$. Accordingly, the gauge transforms $Q$ and $Q'$ respect respectively 
the flags $F$ and $F'$. In the quotient space $E=N/\id N$, we have:

$$
\xymatrix{
E:(\ov{y})\ar[d]^-{U_0} \ar[r]^-{Q} & E:(\ov{y}')\ar[d]^-{P_w} \\
E:(\ov{\vep}) \ar[r]^-{Q'} & E:(\ov{\vep}')
}
$$
The gauge $U_0$ represents geometrically the change of a basis that spans the Harder-Narasimhan flag $\hn_N$
to one that spans the flag $F_N(\la)$ induced by $\la$.

Let $T'=T'_1+\cdots+T'_k$ be the elementary splitting of $T'$. Since $(\ov{\vep})$ repects the flag $F'$, it
will in particular respect the trace of the first element $N_1=\la+\id N$ of the geodesic $\Gamma(N,\la)$, therefore
any lifting of $(\ov{\vep})$ will be a Smith basis of $N_1$ with elementary divisors $T'_1$. Put $T''=T'-T'_1$.
The previous scheme gets thus lifted to the following complete picture.

$$
\xymatrix{
Y_M:(e) \ar[d]_-{t^{T}}\ar[r]^-{t^{T}Qt^{-T}} & Y_{\tM}:(e')\ar[d]^-{t^{T}} \ar@/^1pc/[dr]^-{t^{{T+w(T'_1)}}} & \\
N:(y)\ar@/_2.5pc/[dd]_-{U}\ar[d]^-{U_0}  \ar[r]^-{Q} & N:(y')\ar[d]^-{P_w}\ar[r]^-{t^{w(T'_1)}}  & N_1\ar[d]^-{P_w}& \\
N:(\tilde{y}) \ar[r]^-{Q'} \ar[d]^-{I+tU'}&  N\ar[r]^-{t^{T'_1}} \ar[d]^-{I+tQ'^{-1}U'Q'}& N_1 \ar[d]^-{\tilde{U}\in\Go}&  &\\
N:(\vep) \ar[r]^-{Q'} \ar[d]^-{t^{T'}}& N:(\vep')\ar[d]^-{t^{T'}}\ar[d]^-{t^{T'}} \ar[r]^-{t^{T'_1}} & N_1 \ar[d]^-{t^{T''}}  \\
\la & \la & \la
}
$$
As a result, the elementary divisors of $N_1$ with
respect to the common Birkhoff-Grothendieck trivialisation $\ti{M}$ of $N$ and $N_1$ are not $T+T'_1$ (as with respect to $M$), but 
$T+w(T'_1)$, namely the elements of $T'_1$ have been twisted according to the permutation indexing the
Bruhat cell that contains the matrix $U_0\in \Gc$.

Note that to we have to perform an additional permutation $\sg$ to ensure that
$T+w(T'_1)$ is ordered by increasing values: the resulting ordered diagonal is then $\sg(T+w(T'_1))$. 
\end{proof}

Let $\la\in\ra$ and $M\in \ra^0$ be an arbitrary trivialising lattice of $\la$. 
Let $\Gamma=(\la^0=M,\la^1,\ldots,\la^d=\la)$ be (a normalised representative of) the geodesic
through $[\la],[M]$. Let $(e)$ be a Smith basis of $M$ for $\la$, and let the \del\! $T$ of $\la_M$ in $M$ be 
written as $T=(t_1I_{n_1},\ldots,t_sI_{n_s})$ where $t_1=0<\cdots <t_s$.
Consider the elementary splitting of $T$
\begin{equation}\label{T}
T=T_1+\cdots+T_d\mbox{ where }T_i=(0_{\nu_i},I_{n-\nu_i})
\end{equation}
for a non-decreasing sequence $(\nu_i)$. 
Recall that each partial sum $T_1+\cdots+T_k$ represents the \del\! of $\la^k$ in $M$.
The basis $(e)$ respects the flag $\F_M(\la)$ in the quotient $M/\id M$, and in fact, 
if we let $(e^k)$ be a $\la^k$-basis of the apartment $[\Phi]$ spanned by $(e)$,
then $(e^k)$ respects both flags $F_{\la^k}(M)$ and $F_{\la^k}(\la)$ 
in $\ov{L}^k=\la^k/\id\la^k$ for any $k$.
With the help of lemma~\ref{geodesic:step}, we can construct a Birkhoff-Grothendieck trivialisation
$M_k$ of the $k$-th element $\la^k$ of the geodesic $\Gamma$, which is simultaneously a Birkhoff-Grothendieck trivialisation
of the lattice $\la_M+\id\la^k=\la_M+\id(\la_M+\id^k M)=\la^{k+1}$. At the end of at most $d$ steps, the
lattice $M_d$ is a Birkhoff-Grothendieck trivialisation of $\la_M$,
thus of $\la$. To get the actual type, we only need to subtract $v_{M}(\la)$.

By the way, we have proved the following result.

\begin{prop}\label{type:twisted_del}
Let $M$ be an arbitrary trivialising lattice of $\E$ at $x$. Let 
$T=(t_1I_{n_1},\ldots,t_sI_{n_s})=T_1+\cdots+T_d$ with $t_1<\cdots <t_s$ be the elementary splitting 
of the normalised \del\! $\k-\min_ik_i$ of $\E_x$ in $M$. There exists a sequence of permutations
$w_k\in S_n$ such that the type $T(\E)$ of $\E$ is equal (up to permutation) to 
$-(T_1+w_2(T_2)+\cdots+w_d(T_d))$.
\end{prop}

\subsubsection{The Abacus}
\label{young}

Proposition~\ref{type:twisted_del} corresponds to a combinatorial interpretation of elementary divisors, and
some manipulation of Young tableaux. Let $\appname{M}{\la}{z^{\k}}$ represent the elementary divisors
of $\la$ in $M$ such that $v_M(\la)=0$, and let $\k=T_1+\cdots+T_d$ be the elementary splitting of $\k=(k_1,\ldots,k_n)$. Recall 
that all the sequences $T_i$ have the form $(0_{m_i},I_{n-m_i})$, and that the sequence of the $m_i$ is non-decreasing.
For $w=(w_2,\ldots,w_d)\in S_n^{d-1}$, let $w(\k)=T_1+w_2(T_2)+\cdots+w_d(T_d)$.

Let $Y(\k)$ be the Young tableau containing whose $n$ rows have respective lengths the elements of $\k$ (by decreasing order).
Then we have $T_i=(0_{m_i},I_{n-m_i})$ where $m_i$ is the number of boxes in the $i$-th column. Said otherwise, the sequence 
$(n-m_1,\ldots,n-m_d)$ corresponds to the Young tableau which is dual to $Y(\k)$.

Let us define the {\em abacus} $\ab(\k)$ of
$\k$ as the set of box diagrams obtained from $Y(\k)$ by allowing to move some boxes {\em only vertically} inside the whole corresponding
column of length $n$ (like in a chinese abacus), except in the first
column. As a matter of fact, we could allow to move the boxes in the
first column, but, in this way, we stick to
proposition~\ref{type:twisted_del}. The diagram thus obtained can have
non-adjacent boxes. To any diagram in the abacus, we attach the
sequence $(a_1,\ldots,a_n)$ of number of boxes contained in each of
the $n$ rows. Then we have the following result.

\begin{lemm}\label{abacus}
Let $Y(\k)$ be the Young tableau containing whose $n$ rows have respective lengths the elements of $\k$ (by decreasing order).
The set of sequences $w(\k)$ for $w=(w_2,\ldots,w_d)\in S_n^{d-1}$ is in bijective correspondence with the abacus of $\k$.
Moreover, for any sequence $w(\k)\in\ab(\k)$, we have $\Delta
w(\k)\lsl \Delta \k$ and $i(w(\k))\lsl i(\k)$.
\end{lemm}

\begin{proof}
We will only prove the claim on $i(\k)$, since the other two are clear
by definition. We proceed by induction on the number $d$ of columns in
the Young tableau $Y(\k)$. The Young tableau $Y$ for
$\k=(k_1,\ldots,k_n)$ can be described unequivocally by its dual
$T=(T_1,\ldots,T_d)$. First note that the diagram obtained from $Y$ by
erasing the last column is again a Young tableau $Y'$, corresponding
to the sequence $T'=(T_1,\ldots,T_{d-1})$. Let
$\k'=(k'_1,\ldots,k'_n)$ be the associated sequence. Then we have
$k_i=k'_i$ for $1\lsl i \lsl n-T_d$ and $k_i=k'_i+1$ for $n-T_d+1\lsl
i \lsl n$. Therefore, we get $i(\k)=i(\k')+n-T_d$. 
In fact, an element $N\in \ab(\k)$ given, say, by the permutations $w=(w_2,\ldots,w_d)$
corresponds univoquely to the pair $(N',w_d)$ where $N'\in\ab(\k')$ is
given by the restriction $w'=(w_2,\ldots,w_{d-1})$.

For $d=1$, the claim is clear, for
$i(w(\k))=\abs{\{j\tq k_j=0\}}=i(\k)$. Assume then that for any tableau $Y'=Y(\k')$ with
at most $d-1$ columns, we have $i(w(\k'))\lsl i(\k')$ for $w(\k')\in\ab(\k')$.
Let $Y=Y(\k)$ have $d$ columns. Let $N\in\ab(\k)$ be described by the
number $t_i$ of boxes in the $i$-th row for $1\lsl i\lsl n$, and let
$N'=(t'_1,\ldots,t'_n)$ be the restriction of $N$ to the $d-1$ first
columns. Let $\J=\{i\tq t_i=t'_i+1\}$. Note that $\abs{\J}=T_d$. Then
$i(N)=\sum_{i=1}^n (\max_j t_j -t_i)$. We distinguish two cases:

1) If $\max t_i =\max t'_i=t'_{i_0}$, then we get
\begin{eqnarray*}
i(N)&=&\sum_{i\in\J} (t'_{i_0} -(t'_i+1))+\sum_{i\notin \J} (t'_{i_0}
-t'_i)\\
&=&\sum_{i=1}^n (t'_{i_0} -t'_i)-\abs{\J}=i(N')-T_d
\end{eqnarray*}
By the induction assumption, we have $i(N')\lsl i(\k')$, therefore we
get
$i(N)\lsl i(\k')-T_d=i(\k)-n\lsl i(\k)$.

2) Otherwise, we have $\max t_i =\max t'_i+1=t'_{i_0}+1$. Then we get
\begin{eqnarray*}
i(N)&=&\sum_{i\in\J} (t'_{i_0}+1 -(t'_i+1))+\sum_{i\notin \J} (t'_{i_0}+1
-t'_i)\\
&=&\sum_{i=1}^n (t'_{i_0} -t'_i)+n-T_d\\
&\lsl& i(\k')+n-T_d=i(\k)
\end{eqnarray*}
Therefore the result is established.
\end{proof}

The Birkhoff-Grothendieck trivialisations satisfy thus a local criterion.

\begin{prop}\label{min:dist}
Let $\la\in\ra$ be a lattice. For any Birkhoff-Grothendieck
trivialisation $M$ of $\la$, we have $$d(\la,M)=\min_{\tM\in\ra^0}
d(\la,\tM).$$
Moreover, the triviality index of $\la$ is $i(\la)=\min_{\tM\in\ra^0} i(\k_{\la}(\tM))$.
\end{prop}

\begin{proof}
If $\tM\in \ra^0$ is a trivialisation of $\la$ with elementary divisors $\k$, then, by proposition~\ref{type:twisted_del}, the elementary
divisors $\ti{\k}$ of the Birkhoff-Grothendieck trivialisation $M$ found by the algorithm above 
are up to permutation equal to an element $w(\k)$ of the abacus of
$\k$. Therefore, lemma~\ref{abacus} implies directly the claimed result.
\end{proof}

\subsection{The Permutation Lemma}
\label{type:del}

In the local approach that we are using, the global information on 
the vector bundle is carried by the global form $Y_M$ that sits inside 
any given trivial lattice $M\in\ra^0$. However, any trivial lattice
is not necessarily a Birkhoff-Grothendieck trivialising lattice, therefore
the corresponding elementary divisors do not always give the type of the 
corresponding bundle. Here, we establish the relevant results, that are 
based on the following remarkable lemma.

\begin{lemm}[Permutation lemma]
\label{lem:per}
Let $\k=(k_1,\ldots,k_n)$ be an integer sequence and $P\in \G(\C[[t]])$ a
lattice gauge. Then 
\begin{enumerate}[1)]
 \item (Bolibrukh) there exist a permutation $\sg \in S_n$ and a
lattice gauge $\tP\in \G(\C[[t]])$ such that $$\Pi=t^{-\k}P^{-1}
t^{\k_{\sg}}\tP\in \G(\C[t^{-1}]),$$
where $t^{\k}=\diag(t^{k_1},\ldots,t^{k_n})$ and
$\k_{\sg}=(k_{\sg(1)},\ldots,k_{\sg(n)})$.
\item there exists moreover a lattice gauge $Q\in \Go$ such that 
$t^{\k}\Pi=Qt^{\k}$.
\end{enumerate}
\end{lemm}

We will give a self-contained proof of this result, following for the first item basically 
the same lines as the proof of this lemma given by Ilyashenko and Yakovenko~\cite{IlYa}.
The second part of this lemma is, up to our knowledge, new.

\input{Permutation.tex}

Geometrically, the picture obtained is very evocative. 
$$
\xymatrix{
\la \ar[r]^-{H} \ar[d]_-{t^{T_1}} & 
\la \ar[r]^-{t^{T_1}} \ar[d]_-{\ov{H}_1} & 
M_1 \ar[dl]_-{\Pi_1} \ar[r]^-{t^{T_2}}  & 
M_2 \cdots &
M_{m-1} \ar[r]^-{t^{T_m}} & 
Y\subset M \\ 
\tM_1 \ar[d]_-{t^{T_2}} \ar[r]^-{\ti{H}_1} & 
\tM_1 \ar[d]_-{\ov{H}_2} \ar[r]^-{t^{T_2}} & 
M^2_2 \ar[dl]_-{\Pi_2}\ar[r]^-{t^{T_3}} &  
\cdots \cdots \ar[r]^-{t^{T_m}} & 
Y_2\subset M^2_m & \\ 
\tM_{2} \ar[d]_-{t^{T_3}} \ar[r]^-{\ti{H}_{2}} & 
\tM_{2} \ar[d]_-{\ov{H}_{3}} \ar[r]^-{t^{T_3}} & 
M^3_3 \ar[dl]_-{\Pi_3} \cdots \cdots \ar[r]^-{t^{T_m}}& 
Y_{3}\subset M^{3}_m &  
& \\ 
\vdots & \vdots & &  & &  \\
\tM_{m-1} \ar[d]_-{t^{T_m}} \ar[r]^-{\ti{H}_{m-1}} & 
\tM_{m-1} \ar[d]_-{\ov{H}_{m}} \ar[r]^-{t^{T_m}} & 
Y_{m-1}\subset M^{m-1}_m \ar[dl]_-{\Pi_m} &   
&   
&  \\ 
\tM_m \ar[r]^-{\ti{H}_m} & 
\tilde{Y}\subset \tM_m & 
&  
&   
&
}
$$
The first row corresponds to
a geodesic $\Gamma=(\la,M_1,\ldots,M)$ from $\la$ to a given Birkhoff-Grothendieck trivialisation $M$ . 
This path $\Gamma$ is included in an apartment $[\Psi]$, namely the one spanned by a Birkhoff-Grothendieck
basis $(e)$ of $\la$ corresponding to the trivialisation $M$. By definition, the apartment $[\Psi]$
goes through the global form $Y$ of $M$. The
gauge $H^{-1}$ does not map the geodesic $\Gamma$ onto anything special. However, if we call
$[\Phi]=H^{-1}([\Psi])$ the image of the apartment spanned by $(e)$, the permutation lemma
tells us how to construct a geodesic $\Gamma'$ in $[\Phi]$ whose end point is also a Birkhoff-Grothendieck
trivialisation of $\la$. Lemma~\ref{iter:step} gives the step-by-step modification of the 
geodesic $\Gamma$. Row $i$ of the diagram corresponds indeed to a partial geodesic 
$\Gamma_i=(\tM_i,\ldots,M^i_m)$ whose end-point is a Birkhoff-Grothendieck trivialisation of
the $i$-th element $M_i$ of the geodesic $\Gamma'=(\la,\tM_1,\ldots,\tM_m)$. Even if the end-point 
$\tM_m$ is a Birkhoff-Grothendieck trivialisation of $\la$, note that the apartment $[\Phi]$ does not
contain the global form $\ti{Y}$ of $\tM_m$, and that we still need the gauge transform $\ti{H}_m$
to obtain it.

Since a permutation leaves a frame unchanged, we can deduce the
following result.

\begin{thm}
\label{per}
Let $\E$ be a holomorphic vector
bundle over $X$, and let $\la=\E_x\in \ra$ be its stalk at $x\in X$.
For any apartment $[\Phi]$ in the Bruhat-Tits building $\B$ at $x$ such
that $[\la]\in [\Phi]$, there exists a Birkhoff-Grothendieck
trivialisation of $\la$ in $[\Phi]$.
\end{thm}

\begin{proof}
Let $(e)$ be a Birkhoff-Grothendieck basis of $\la$, and $M$ be a 
Birkhoff-Grothendieck trivialisation of $\la$. 
Let $(\vep)$ be a basis of the lattice $\la$ which spans
the apartment $[\Phi]$. Since $[\Phi]$ is invariant under $S_n$, we can
assume that the matrix $P\in \Go$ of the basis change from $(\vep)$ to $(e)$ has
invertible principal leading minors. According to the permutation lemma, 
there exists a matrix $\tP\in \Go$
such that $\Pi=z^{-K}P^{-1}z^{K}\tP\in \G(\C[z^{-1}])$.
The gauge $\Pi$ sends the basis of global sections $(\sg)=(z^{K}e)$ of
the Birkhoff-Grothendieck trivialisation of $\E$, given
at $x$ by $M$, into a basis $(\te)$ of $\tM$. Since $\Pi$ is a monopole, the
basis $(\te)$ is also a global basis of sections, but spans another
trivialising bundle, namely $\F=\E^{\tM}$. 
Therefore the arbitrary apartment $[\Phi]$ spanned by $(\vep)$ indeed contains a
trivial bundle. Now the matrix
$\ov{H}=z^{K}\Pi$ admits a right factorisation $\ov{H}=Qz^{K}$. As a
consequence, if we let $(\ti{\vep})$ be the basis of $\la$ obtained from
$(e)$ by the matrix $Q$, then $z^{K}(\ti{\vep})$ is also a basis of
$Y_{\tM}$. The following scheme sums up the 
situation.
$$
\xymatrix{\relax
 \la:(\ti{\vep}) \ar@/^1pc/[rrd]^-{z^{K}}& & \\
\la:(e) \ar[u]^{Q}\ar[r]^-{z^{K}} & Y_M\subset M \ar[r]^-{\Pi}
& Y_{\tM}\subset \tM \\
\la:(\vep) \ar[u]^-{P} \ar[r]^-{z^{K}} & \tM
\ar[ur]^-{\tP} & 
}
$$
Therefore the lattice $\tM$ is also a Birkhoff-Grothendieck
trivialisation of $\la$.
\end{proof}

\begin{cor}
\label{cor:formChange}
Let $M$ be a Birkhoff-Grothendieck trivialisation of 
the lattice $\la=\E_x$, and $K=(k_1,\ldots,k_n)$ the type of the
corresponding bundle $\E$. For any form $Y$ in $\la$, the lattice
$\tM=\Psi_Y(\F_{\la}(M),K)$ is also a Birkhoff-Grothendieck trivialisation
of $\la$.
\end{cor}

\begin{proof}
Let $(e)$ be a Smith basis of $\la$ for $M$. 
The lattice $\tM$ is spanned by $z^K(e_Y)$ where $(e_Y)$ is the 
$Y$-basis of $(e)$. The only thing to check is that the gauge from $(e)$ to $(e_Y)$ has
invertible principal leading minors, but this is obvious since 
the gauge is tangent to $I$.
\end{proof}

The permutation lemma is in fact a sort of converse to the
Birkhoff-Grothendieck theorem. Indeed, this theorem asserts 
that for any lattice $\la$ in the
Bruhat-Tits building at infinity
there exists a trivialising lattice $\ti{M}$ such that there is a
Smith basis for $\la$ sitting inside
the global form $Y_{\ti{M}}$. The problem then amounts to, given a
lattice gauge $P$ and a diagonal $K$,
find $Q$ and $\ti{K}$ such that $z^{-\ti{K}}Qz^{K}P \in \G(\C[z])$,
whereas in the permutation lemma, the
input data would be the matrix $Q$ and the diagonal $\ti{K}$.
Schematically, the picture would be like this
$$
\xymatrix{
\la \ar[r]^{z^{K}} \ar[d]_-{Q^{-1}} & M \ar[r]^-{P} & Y_M \subset M  \\
\la \ar[r]^-{z^{\tilde{K}}} & \tilde{M}\supset Y_{\tM} \ar[ur]_-{\Pi} &
}
$$

\section{Local Meromorphic Connections}
\label{lmc}

Let $D=\mathrm{Der}_{\C}(K)$ be the $K$-vector space of dimension 1 of
$\C$-derivations of $K$ and  $\om=\om^1_{\C}(K)$ the dual composed
of differentials of $K$. The valuation $v$ extends naturally to these
spaces by the formul{\ae} $v(\vartheta)=v(f)$ and $v(\omega)=v(g)$
if $\vartheta = f \ddz{}$ and $\omega=g \: dz$ for
any uniformising parameter $z$ of $K$. The space $\om$ is naturally
filtered by the rank 1 free $\h$-modules $\om(k)=\{\omega \in \om \tq
v(\omega) \gsl -k \}$.

Let $V$ be a $K$-vector space  of finite dimension $n$ and let
$\om(V)=V\otimes_{K}\om^1_{\C}(K)$. We fix a {\em meromorphic
connection} $\na$ on $V$. This is an additive map $\appto{\na
}{V}{\om(V)}$ satisfying the Leibniz rule $$\na(fv)=v\otimes df+f\na
v\text{ for all }f \in K\text{ and all }v \in V.$$
For any basis $(e)=(e_1,\ldots,e_n)$ of $V$, the {\em matrix $\mat(\na,(e))$
of the connection $\na$ in the basis $(e)$} is the
matrix $A=(A_{ij})\in \m(\om)$ such that
$$\na e_{j}=-\sum^n_{i=1} e_i\otimes A_{ij}\text{ for all }j~=~1,
\ldots, n.$$  If the matrix $P=\mat(\idd_V,(\varepsilon),(e))\in \G(K)$ is
the basis change from $(e)$ to any other basis $(\varepsilon)$, then the matrix of $\na$ in
$(\vep)$ is given by the {\it gauge transform} of $A$
\begin{equation} \label{jauge} A_{[P]}=P^{-1}AP-P^{-1}dP.
\end{equation}
For any derivation $\tau\in \de$, the contraction of $\na$ with $\tau$
induces a differential operator $\nat$ on $V$. The connection $\na$ 
is {\em regular} whenever the set of {\em logarithmic}
lattices $$\ra_{\log}= \{ \la \in \ra \tq \na(\la)\subset
\la\otimes_{\h}\om(1)\}$$ is non-empty. For any logarithmic lattice 
$\la\in\ra_{\log}$, the connection $\na$ induces a well-defined 
residue endomorphism $\r_{\la}\na\in \End_{\C}(\la/\id\la)$.
Note that, since the set $\ra_{\log}$ is closed under homothety and 
module sums (\cite{Cor6}, lemma 2.5), it induces a {\em geodesically convex} subset of the Bruhat-Tits building: 
if $L,L'\in\ra_{\log}$, then $\Gamma(L,L')\subset \ra_{\log}$.

\subsection{The Deligne Lattice}
\label{deligne}

As is well known, the choice of a matrix logarithm corresponds to
fixing a special lattice
in the space $V$. More precisely, let $V^{\na}\subset \tens{V}{K}{H}$
be the $\C$-vector space of
horizontal sections on any Picard-Vessiot extension $H$ of $K$. Let
$g=g_s g_u \in \End(V^{\na})$ be the multiplicative Jordan decomposition of the
corresponding local monodromy map. Then the logarithm of the unipotent
part $g_u$ is canonically defined (by the Taylor expansion formula for $\log(1+x)$),
but there are several ways to define the logarithm of the semi-simple
part $g_s$. Namely, one must fix a branch of the complex logarithm for every {\em
distinct} eigenvalue of $g_s$.

A classical result (variously attributed to Deligne, Manin..., see~\cite{Sa}) says
that this choice uniquely defines a lattice in $V$. In Deligne's terms, for any
section $\sg$ of $\C\longrightarrow \C/\Z$, there is a unique
logarithmic lattice $\Delta_{\sg}$
such that the eigenvalues of the residue map $\r_{\Delta_{\sg}}\na$
are in the image $\im \sg$ of $\sg$. As a habit, one usually takes $\re(\im \sigma) \subset [0,1[$. 
In fact, such a habit is not as arbitrary as it seems.

\begin{prop}
\label{Deligne:hol}
Assume that the connection $\na$ admits an {\em apparent singularity}
({\it i.~e.} the monodromy
map is trivial). Then the matrix $\mat(\na,(e))$ is {\em holomorphic}
if and only if
the lattice spanned by $(e)$ is equal to the Deligne lattice $\Delta$
attached to $\re(\im \sigma) \subset [0,1[$.
\end{prop}

\begin{proof}
Since the monodromy map is trivial, its normalised logarithm with respect to $\Delta$
is 0. Hence, there is a basis of $\Delta$ where  the connection 
has matrix 0. In any other basis $(e)$ of $\Delta$, 
the connection has matrix $A=P^{-1}dP$ with $P\in \Go$,
which is holomorphic.
Let $M$ be another lattice, and let $(e)$ be a Smith basis of $\Delta$
for $M$. Then the matrix in a basis of $M$ is given by the gauge equation
$$\ti{A}=z^{-K}Az^K-z^{-K}d(z^K)=(A_{ij}z^{k_j-k_i})-K\frac{dz}{z}.$$
The non-zero diagonal terms of
the matrix $K$ of elementary divisors of $M$ give necessarily rise to
a pole of order 1 in $\ti{A}$.
Therefore, $\Delta$ is the only lattice where the connection has a
holomorphic matrix.
\end{proof}
As a result, we will call $\Delta$ {\em the} Deligne lattice of $V$.

\subsubsection{Birkhoff Forms}
\label{BF}

According to a very classical result (see {\em e. g.}~\cite{Ga}, p. 150) 
if $$\Omega=\mat(\na,(e))=\sum_{k\gsl 0}A_kz^k\frac{dz}{z}$$ is the series expansion in $z$ of
the matrix of $\na$ in a basis $(e)$ of $\Delta$, the gauge $P=\sum_{k\gsl 0}P_kz^k\in \Go$ defined recursively by
\begin{equation}\label{G}
\left\{\begin{array}{l}
          P_0=I\\
          P_k=\Phi_{A_0,A_0-kI}^{-1}(Q_k)\mbox{ where }Q_k=\sum_{i=1}^kA_iP_{k-i}
         \end{array}
\right.
\end{equation}
transforms $\Omega$ into $A_0dz/z$. Here we put $\Phi_{U,V}(X)=XU-VX$.
Recall that the map $\Phi_{U,V}$ is an automorphism of $\gl(\C)$ when the 
spectra of $U$ and $V$ are disjoint. The gauge $P$ thus defined is uniquely determined; 
moreover, the set of bases where $\na$ has matrix $L\frac{dz}{z}$ 
where $L\in\m(\C)$ is a constant matrix spans a form $\Upsilon_z$ of $\Delta$, 
that we call the {\em Birkhoff form} of the Deligne lattice $\Delta$. The gauge transform
$P$ sends in fact the basis $(e)$ to its $\Upsilon_z$-basis, that we denote here for simplicity $(e_z)$.

As it results from the proof of proposition~\ref{Deligne:hol}, when the singularity is apparent, the Birkhoff form is
uniquely defined. Otherwise, however, the form $\Upsilon_z$ depends on the choice of the local coordinate $z$. Two Birkhoff forms 
are nevertheless canonically isomorphic.
\begin{lemm}\label{B:form}
Let $z,t$ be two local coordinates, and let $\alpha \in \h^*$ such that $z=\alpha t$.
Let $P_z$ and $P_t$ be the gauge transforms that send $(e)$ to $(e_z)$ and $(e_t)$ respectively.
There is a {\em unique} gauge transform $\ti{P}$ that sends $(e_z)$ to $(e_t)$.
\end{lemm}

\begin{proof}
One has $\frac{dz}{z}=u\frac{dt}{t}$ with $u=1+\frac{\theta_t \alpha}{\alpha}$ where $\theta_t=t\frac{d}{dt}$. 
Put $u=\sum_{i=0}^{\infty}u_it^i$. Accordingly, the matrix of the connection in $(e_z)$ satisfies
$$\mat(\na,(e_z))=A_0\frac{dz}{z}=A_0\left(\sum_{i=0}^{\infty}u_it^i\right)\frac{dt}{t}.$$
There exists therefore a uniquely defined gauge transform $\ti{P}=\sum_{i=0}^{\infty}\ti{P}_it^i$ 
that transforms the expression $A_0dz/z$ into $A_0dt/t$, as explained in the following scheme.
$$
\xymatrix{
\Omega \ar[r]^-{P_z} \ar[d]_-{P_t} & A_0\frac{dz}{z}=\sum_{i=0}^{\infty}u_iA_0t^i\frac{dt}{t} \ar[dl]_-{\ti{P}} \\
A_0\frac{dt}{t} & 
}
$$
The matrix series $\ti{P}$ is determined recursively by the equations~(\ref{G}) applied to the series 
$\sum_{i=0}^{\infty}A_0u_it^i$. The coefficients $\ti{P}_i$ are even {\em polynomials in $A_0$}, defined by the following
induction rule
\begin{equation*}
\left\{\begin{array}{l}
          \ti{P}_0=I\\
          \ti{P}_k=\frac{1}{k}\sum_{i=1}^ku_iA_0\ti{P}_{k-i}
         \end{array}
\right.
\end{equation*}
\end{proof}

\subsection{Logarithmic Lattices and Stable Flags}
\label{stable}

When two lattices $\la,M$ are adjacent, all the relevant information
on $M$ can be retrieved from the quotient $M/\id\la$. This is also true
in presence of a connection.

\begin{lemm}
\label{adj:log}
Let $\la\in \ra_{\log}$ be a logarithmic lattice. For any adjacent lattice
$M\in [\id\la,\la]$, we have $M\in \ra_{\log}$ if and only if 
$M/\id\la$ is $\r_{\la}\na$-stable.
\end{lemm}

\begin{proof}
In any basis $(e)$ of $\la$ such that the images of the first $m=\dim W$ vectors span 
$W=M/\id\la$, the connection matrix $\Omega=\mat(\na,(e))$ has a residue of the form 
$\left(\begin{array}{cc} A & B \\
        0 & D
       \end{array}
\right)\in \m(\C)$, where $A\in \mathrm{M}_{m}(\C)$. Putting $T=\diag(0_m,I_{n-m})$, 
the basis $(\vep)=z^T(e)$ spans $M$. It is then straightforward that 
the matrix $z^{-T}\Omega z^T-T\dzz$ of $\na$ in $(\vep)$ has a simple pole.
\end{proof}

When the lattices are further apart, this correspondence fails. However,
there is also a complete description of the logarithmic lattices as follows.
Let $\Delta$ be the Deligne lattice, and let $\dt_{\Delta}=\r_{\Delta}\na$ be the
residue $\C$-endomorphism on $D=\Delta/\id\Delta$. Let $\Upsilon$ be the Birkhoff
form of $\Delta$ attached to a uniformising parameter $z$.
Logarithmic lattices can then be characterised as stable flags 
(as already remarked by Sabbah~\cite{Sa}, th. 1.1).

\begin{prop}
\label{bij:lattices}
The set $\ra_{\log}$ of logarithmic lattices is in bijection with
the subset $W_0(\Upsilon)$ of $W(\Upsilon)$ defined
by
$$W_0(\Upsilon)=\{(F,\k)\in W(\Upsilon) \tq F \textrm{ stabilised by }\dt_{\Delta}\}.$$
\end{prop}

\begin{proof}
According to a classical, although not so well known, result (which can be found 
for instance in \cite{BV,Bo1,Cth}), a lattice $\la\in
\ra$ is logarithmic if and only if
\begin{enumerate}[i)]
\item There exists a basis $(e)$ of $\Upsilon$ such that $(z^{K}e)$ is a
basis of $\la$, with $K=\k_{\Delta}(\la)$,\label{slati}
\item $z^{-K}Lz^K \in \m(\h)$, where $L=\mat(\nat,(e))$.\label{slatii}
\end{enumerate}
It results from (\ref{slatii}) that in this case, the matrix $L$ is $K$-parabolic.
Since the flag $\F_{\Delta}(\la)$ induced by $\la$ on
$D=\Delta/\id\Delta$ is spanned by the images of the basis $(e)$ in $D$, 
it is stable under $\dt_{\Delta}$. Conversely, it
is simply a matter of computation to show that any lattice in the $\Upsilon$-fibre of 
a $\dt_{\Delta}$-stable flag of $D$ is logarithmic.
\end{proof}

A difference between our result and Sabbah's is that he only states this result
as an equivalence of categories between the set of stable filtrations of $D$ and the
logarithmic lattices, whereas we give the explicit correspondence based on the 
lifting of $D$ to a Birkhoff form. Although it would seem that the previous result has little value to effectively 
determine all logarithmic lattices, it is always possible to determine them in
finite terms. 

\begin{lemm}\label{poly:latt}
Let $M\in \ra_{\log}$ and let $(F,K)=\Phi_{\Delta}(M)$. 
Let $Y$ be a form of $\Delta$, and let $(e)$ be a basis of $Y$ respecting 
the flag $F$. Fix a coordinate $z$, and 
let $P=I+P_1z+\cdots $ be the gauge from $(e)$ to its $\Upsilon_z$-basis 
$(e_z)$. Then the Laurent
polynomial gauge $Q\in\gl(\C[z,z^{-1}])$ defined by
$$Q=(I+\cdots +P_{d-1}z^{d-1})z^K\mbox{ where }d=d(\Delta,M)$$
sends the basis $(e)$ of $\Delta$ to a basis of $M$.
\end{lemm}

\begin{proof}
This is an almost direct consequence of lemma~\ref{poly:form}.
\end{proof}

Note that the polynomial gauge $Q$ can be explicitly computed from
formula~(\ref{G}). On the other hand, one can also explicitly describe
the set $W_0(\Upsilon)$. For a linear map $f\in \End(\C^n)$, let $\mf{B}_f$
be the set of complete flags that are stable under $f$, and say that an apartment
$[\Phi]$ is a {\em diagonalising apartment} of $f$ if the frame $\Phi$ 
is composed of eigenlines of $f$. Then we have the following.

\begin{lemm}
Let $\dt_{\Delta}=\mathfrak{d}+\mathfrak{n}$ be the additive Jordan
decomposition of the residue map $\dt_{\Delta}=\r_{\Delta}\na$.
The pair $(F,K)\in W(\Upsilon)$ is an element of $W_0(\Upsilon)$ if and only
if $F$ admits a complete flag refinement
$\hat{F}$ such $\hat{F}\in \mf{B}_{\mf{n}}$ and there is a
diagonalising apartment $[\Phi]$ for $\mf{d}$ that respects
the flag $\hat{F}$.
\end{lemm}

\begin{proof}
A flag $F$ is $\dt$-stable if and only if it is stable under
both $\mathfrak{d}$ and $\mathfrak{n}$. It is known that 
$F$ is stable under $\mathfrak{d}$ if and only if every
component $F_i$ of $F$ is a direct sum of $\mathfrak{d}$-stable lines,
and under $\mathfrak{n}$ if and only if it admits a complete
flag refinement $\hat{F}$ that belongs to the flag subvariety 
$\mathfrak{B}_{\mathfrak{n}}$ of complete flags which are preserved by the action
of $\mathfrak{n}$.
\end{proof}

\section{The Riemann-Hilbert Problem}
\label{RH}

This problem is by now very well-known, so we will just state the
necessary notations and definitions, and
refer to the classical paper of Bolibrukh \cite{Bo1} and to the
account he gives of the construction of the
Deligne bundle (see also \cite{Sa} and \cite{IlYa}).

Let $\dd=\{s_1,\ldots,s_p\}$ a prescribed set of singular points,
$z_0\notin \dd$ be an arbitrary base point,
and let a representation
\begin{equation}\label{chi}
\appto{\chi}{\pi_1(X\backslash \dd, z_0)}{\Gc}.
\end{equation}
The Riemann-Hilbert problem asks informally for a linear differential system
having $\chi$ as monodromy representation. In the terms used in this paper, it
asks for a regular meromorphic connection $\na$ with singular set $\dd$ and
monodromy $\chi$ on a holomorphic vector bundle $\E$. If the bundle
is required to be logarithmic with respect to $\na$ one speaks 
of a {\em weak solution} to RH. In its strongest form, the Riemann-Hilbert problem asks 
for a differential system $\pr{Y}=A(z)Y$ having simple poles on $\dd$ as only
singularities, and whose monodromy representation is globally
conjugate to $\chi$. This amounts to asking for a weak solution $(\E,\na)$
which is moreover {\em trivial}.

\subsection{The R\"ohrl-Deligne Construction}
\label{RD}

We briefly recall H. R\"ohrl's construction (as presented for instance 
in~\cite{Bo1,BMM}). Let $\u=(U_i)_{i\in I}$
be a finite open cover of $X^*=X\backslash \dd$ by
connected and simply connected open subsets $U_i\subset X^*$ such that
their intersection has the same property,
and all triple intersections are empty. Let arbitrary points $z_i\in
U_i$ and $z_{ij}\in U_i\cap U_j$, and paths
$\appto{\ga_i}{z_0}{z_i}$ and $\appto{\ga_{ij}}{z_i}{z_{ij}}$, so that
$\dt_{ij}=\ga_i\ga_{ij}\ga_{ji}^{-1}\ga_j^{-1}$ is a
positively-oriented loop around $z_i$ having winding number 1.
Then the cocycle $g=(g_{ij})$ defined over $\u$ by the constant
functions $g_{ij}=\chi([\dt_{ij}])$
defines a flat vector bundle $\F$ over $X^*$. Define the connection
$\na$ over $U_i$ by the $(0)$
matrix in the basis of sections corresponding to the cocycle $g$. The
$\na$-horizontal sections of $\F$
have by construction the prescribed monodromy behaviour. This solves
what we called the {\em topological} Riemann-Hilbert problem in our
introduction.

Add now a small neighbourhood $D$ of each singular point $s\in \dd$ to
the cover $\u$, in such a way that
$D\backslash \{s\}$ is covered by $k$ pairwise overlapping sectors
$\Sg_1=D\cap U_{j_1},\ldots,\Sg_k=D\cap U_{j_k}$.
On an arbitrarily chosen sector among the $\Sg_i$, say, $\Sg_1$, let
$\tilde{g}_{s1}=z^{L}$ where $z$ is a
local coordinate at $s$ and $L=\frac{1}{2i\pi}\log \chi(\dt)$
normalised with eigenvalues
having their real part in the interval $[0,1[$. Since the open subset
$\Sg_1$ only intersects $\Sg_2$ and $\Sg_k$,
the only necessary cocycle relations to satisfy are
$\tilde{g}_{s2}=\tilde{g}_{s1}g_{12}$ and
$\tilde{g}_{sk}=\tilde{g}_{s1}g_{1k}$,
that we take as {\em definition} of the cocycle elements
$\tilde{g}_{s2}$ and $\tilde{g}_{sk}$. Define in this way
the remaining elements of the cocycle $\ti{g}$ on on $D\cap U_{j_i}$.
By construction, the result defines a holomorphic vector
bundle $\D$ on the whole of $X$, and the connection $\na$ can be
extended as $L\frac{dz}{z}$ in the
basis of sections $(\sg)$ of $\D$ over $D$ chosen to construct $\tilde{g}_{s1}$.
The pair $(\D,\na)$ is called the {\em Deligne bundle} of $\chi$. This
construction solves simultaneously the meromorphic and the weak Riemann-Hilbert
problem.

\begin{rem}\label{bf}
The basis $(\sg)$ is, in our terms, a basis of the Birkhoff form
attached to the coordinate $z$ at $s$.
\end{rem}

\subsection{Weak and Strong Solutions}
\label{RH:weak}

The Riemann-Hilbert problem can be seen as involving
three different levels. The topological level is only governed by the
(analytic) monodromy around the prescribed singular set. The meromorphic
level is essentially based on the solution of the {\em local} inverse
problem. The third one, that we call {\em holomorphic} is global and
asks for the existence of a trivial holomorphic vector bundle.
In fact, separating these three aspects is not so easy to do, because
the R\"ohrl-Deligne
construction in fact yields a particular {\em holomorphic} vector
bundle $\E$ with a
connection $\na$ that already respects the holomorphic prescribed
behaviour.

What makes the strong Riemann-Hilbert problem a difficult one is
precisely this third level.
The local meromorphic invariants added to the topological solution of the
inverse monodromy specify up to {\em meromorphic equivalence class}
the connection $\na$ on $X$. In this respect, the natural category to
state this construction is not the category of holomorphic vector
bundles with meromorphic connections, but the meromorphic vector bundles, that
is, pairs $(\V,\na)$ where $\V$ is locally (but in fact globally)
isomorphic to $\mx^n$. This is why we call the second step
{\em meromorphic}. The Riemann-Hilbert problem with the given data
solved here corresponds to the {\em very weak} Riemann-Hilbert problem (as coined by
Sabbah~\cite{Sa}): any subsheaf $\F$ of locally free $\h_X$-modules
contained in the (trivial) meromophic bundle $\V$ is endowed naturally with the
connection $\na$, and therefore is a holomorphic vector bundle with a regular connection
having the prescribed monodromy.
As stated by the next result (and is otherwise well known),
all solutions to the weak problem are obtained as local modifications of
the Deligne bundle.

\begin{prop}
\label{weak:sol}
Let $\appto{\ti{\pi}}{\ti{E}}{X}$ and
$\appto{\ti{\na}}{\ti{\E}}{\tens{\ti{\E}}{\h}{\om}}$ be a
weak solution to the Riemann-Hilbert problem. Then there exist a finite
set $S\subset X$, and local lattices $M_x$ for $x\in S$ such that the pair
$(\ti{\E},\ti{\na})$ is holomorphically isomorphic to $(\D^M,\na)$.
\end{prop}

The last step of the strong Riemann-Hilbert problem consists in
searching the set of holomorphic vector bundles endowed with the connection $\na$ for a bundle which at
the same time has the required holomorphic invariants and is
holomorphically trivial. A negative answer requires to know
all the holomorphic vector bundles with this prescribed
logarithmic property.  Note that up to this point, the discussion
presented in this section holds over an arbitrary compact Riemann surface.

\subsubsection{Plemelj's Theorem}
\label{plemelj}

In 1908, the Slovenian mathematician J. Plemelj (see \cite{Pl}) proved a first
version of the strong Riemann-Hilbert problem, under the assumption
that at least one monodromy is diagonalisable. Whereas his first proof
used an analytic approach (Fredholm integrals) to construct the
actual matrix of solutions, to thence deduce the differential system and
prove that it has only simple poles, the general framework of vector bundles
recalled so far allows to establish this fact in an amazingly concise way.

\begin{thm}[Plemelj]
If one of the elementary monodromy maps from representation
$\appto{\chi}{\pi_1(X\backslash \dd,z_0)}{\Gc}$ is
diagonalisable, then the Riemann-Hilbert problem has a strong
solution.
\end{thm}

\begin{proof}
Let $(\D,\na)$ be the R\"ohrl-Deligne bundle attached to the
representation $\chi$. Let, say $G=\chi(\ga)$ around $s\in\dd$, be
diagonalisable. Let $\Upsilon$ be a Birkhoff form at $s$, and let
$(e)$ be a basis of $\Upsilon$ where $G$ is diagonal. According to condition
\ref{slatii}) in section \ref{stable}, the whole apartment $[\Phi]$ spanned
by $(e)$ consists of logarithmic lattices, whereas theorem ~\ref{per} 
implies that $[\Phi]$ contains a trivialising lattice $M$.
The vector bundle $\D^M$ is therefore both logarithmic and trivial.
\end{proof}

\begin{rem}
Here we have a solution by modifying the Deligne bundle only at one
point. Note that the lattice $M$ corresponds to a Birkhoff-Grothendieck
trivialisation of $\D$ (see theorem~\ref{BG:logTriv} below). Also note that this result also holds replacing
$\D$ with any other weak solution to Riemann-Hilbert.
\end{rem}

\subsubsection{Trivialisations of Weak Solutions}
\label{Triv:weak}

Let $\E$ be a weak solution of the Riemann-Hilbert problem, and 
let $\F$ be a trivialisation of $\E$ at $x\notin \dd$. In a global basis of 
sections $(e)$ of the bundle $\F$, the connection $\na$ is expressed by
the matrix of global meromorphic 1-forms $\Omega$, which has a simple pole
at every $s\in \dd$, and an {\em a priori} uncontrolled pole at $x$.
Assuming for simplicity that $x\notin \dd$ is the point at infinity 
$\infty\in \P$, there exist
matrices $A_i \in \m(\C)$ for $1\lsl i\lsl p$ and a matrix
$$B(z)=B_0+\cdots+B_t z^t$$ such that the connection has the following matrix
$$\Omega=\left(\sum_{i=1}^p \frac{A_i}{z-s_i}+B(z)\right)dz.$$
The most surprising consequence of the permutation lemma,
as we state it, concerns the analytic invariants of the weak solutions
to the Riemann-Hilbert problem.

\begin{thm}
\label{BG:logTriv}
Let $\E$ be a weak solution to the Riemann-Hilbert problem for
$\chi$. Then, for any $x\not\in \dd$, there exists a
Birkhoff-Grothendieck trivialisation $\F$ of $\E$ at $x$ which is also
{\em logarithmic} at $x$. Let $Y=\Gamma(X,\F)$ and let $\psi_s=\r_s^{\F}\na\in \End(Y)$.
Then we have the following.
\begin{enumerate}
\item The map $\Psi=\sum_{s\in \dd} \psi_s=-\r_x\na$ is semi-simple, and has
integer eigenvalues, which are equal to the type of the bundle $\E$.
\item The image of the Harder-Narasimhan filtration of $\E$ in $Y$ is
equal to the flag induced by the eigenspaces of $\Psi$ ordered by
increasing values. 
\end{enumerate}
\end{thm}

\begin{proof}
If $x\notin \dd$, the monodromy at $x$ is trivial, and the stalk
$\E_x$ of $\E$ coincides with $\D_x$. The Birkhoff form $\Upsilon$ of $D$
(which is then unique) is equal to the space $V^{\na}$ of horizontal sections at $x$. 
All flags in $D=\D_x/\id_x\D_x$ are stable under $\r_x^{\D}\na=0$. 
According to corollary~\ref{cor:formChange}, the $\Upsilon$-lifting 
of the flag induced by any Birkhoff-Grothendieck
trivialisation of $\E$ at $x$ is a 
{\em logarithmic Birkhoff-Grothendieck trivialisation}
of $\E$ at $x$. In a global basis of sections $(e)$ of 
$\F$, the connection has the following matrix
\begin{eqnarray}
A&=&\sum_{s\in \dd\minus{\infty}} \frac{A_s}{z-s}+\frac{B}{z-x}\mbox{ where
}B=-\sum_{s\in \dd} A_s
\mbox{ if }x\neq \infty\label{FSx1}\\
&=&\sum_{s\in \dd} \frac{A_s}{z-s}
\mbox{ if }x=\infty\notin \dd\label{FSx2}
\end{eqnarray}
since $\na$ has no other singularities
outside $\dd\cup\{x\}$. The eigenvalues of $-B=\sum_{s\in \dd} A_s$ are therefore equal to
the type of $\E$, and the Harder-Narasimhan filtration is defined by
the blocks of equal eigenvalues ordered by increasing values. 
\end{proof}

As a consequence, we deduce the following new sufficient condition for the
solubility of the strong Riemann-Hilbert problem.

\begin{cor}
\label{local:log}
Let $\E\subset H$ be a holomorphic vector bundle in $(\V,\na)$, and let $\D$ be
the Deligne lattice of $\V$. Let $x\in X$, such that $\E_x=\D_x$. 
If the flag $F$ induced in $D=\D_x/\id_x\D_x$
by the stalk $\F_x$ of a Birkhoff-Grothendieck trivialisation $\F$ of $\E$ at $x$
is stable under the residue map $\r_x^{\D}\na\in \End(D)$, then there exists a 
Birkhoff-Grothendieck trivialisation $\ti{\F}$ of $\E$ at $x$ which is moreover
{\em logarithmic} at $x$.
\end{cor}

\begin{proof}
Let $\tM$ be the $\Upsilon$-lifting of the flag $F$, where $\Upsilon$ is a Birkhoff form of
the local stalk $\D_x$ of the Deligne bundle at $x$. According to proposition~\ref{bij:lattices},
the lattice $\tM$ is logarithmic, and by the permutation lemma, it is
a Birkhoff-Grothendieck trivialising lattice. Therefore, the bundle $\E^{\tM}$
satisfies the conclusions of the corollary.
\end{proof}

At this point, we would like to sum up our findings about trivial bundles in the following
proposition.

\begin{prop}\label{triv:resume}
Let $\F\in H_0$ be a trivial bundle in $\V$, and let $Y=\Gamma(X,\F)$ be the $\C$-vector space
of global sections. Let $x\in X$, and $\E\in H$ such that $(\F,\E)\in R_x$.
\begin{enumerate}[i)]
\item $Y$ admits a well-defined flag $\hn$ induced by the Harder-Narasimhan filtration of $\E$.\label{t1}
\item If $\F$ is a Birkhoff-Grothendieck trivialisation of $\E$ at $x$, then the flag $\hn$ is obtained
from a Smith basis of $Y$ for the stalk $\E_x$, according to the elementary divisors $\k_{\E_x}(\F_x)$,
which give moreover the {\em type} $T(\E)$ of the bundle $\E$.\label{t2}
\item If $\F$ is additionally {\em logarithmic} at $x$, and the stalk $\E_x$ coincides with the Deligne lattice
$\D_x$, then the type $T(\E)$ is given by the integer parts of the eigenvalues of the residue $\r_x^{\F}\na\in \End(Y)$,
that is, of the {\em exponents} of $\na$ on $\F$ at $x$.\label{t3}
\item Finally, if $\E\in \RH_{\chi}$ is moreover a weak solution to Riemann-Hilbert, then the following relation holds
$$\sum_{x\in X}\r_x^{\F}\na=0.$$\label{t4}
\end{enumerate}
When $(\E,\F)$ satisfy~\ref{t1}) to~\ref{t4}), we say that they form a {\em weak RH-pair} at $x$.
\end{prop}

Let $(\E,\F)$ be a weak RH-pair at $x\notin \dd$. Let $(\sg)$ be any basis of $Y=\Gamma(X,\F)$.
In $(\sg)$, the connection has a matrix of the form~(\ref{FS}). The identification of $Y$ 
to $\C^n$ by means of $(\sg)$ endows $\C^n$ with $p+1$ linear maps $\psi_s$ for $s\in \dd^*=\dd\cup\{x\}$, 
that we can identify with the matrices $\ti{L}_s$ for $s\in \dd$ and $-\sum_{s\in \dd} \ti{L}_s$
for $s=x$. With these notations, we set the following definition.

\begin{defn}\label{lfm}
The space $\C^n$, endowed with the maps $\psi_s$ for $s\in \dd^*$ 
is called a {\em linear Fuchsian model} of $\E$. 
\end{defn}

With this notion, we can reduce some questions about vector bundles to linear algebra statements.
For instance we can give the following computable version of a criterion for the reducibility
of the triviality index (originally appearing in Sabbah~\cite{Sa}, cor. ?), that we state here only
for the case of a logarithmic modification.

\begin{cor}\label{index:red}
Let $\E\in \RH_{\chi}$ be a weak solution, and consider a linear Fuchsian model at $x\notin\dd$, 
given by $p$ matrices $A_s$ for $s\in \dd$ such that 
$$\sum_{s\in\dd}A_s=\k=\diag(\kk_1I_{n_1},\ldots,\kk_sI_{n_s})$$
where the integers $\kk_i$ satisfy $\kk_i > \kk_{i+1}$, in such a way that
the flag $\hn$ is the flag $0=F_0\subset F_1\subset\cdots\subset F_s=\C^n$ 
having signature $(n_1,\ldots,n_s)$ in the canonical basis of $\C^n$. 
There exists a weak solution $\E'$ which is adjacent to $\E$ at $s\in X$ if and only
if there exists an $A_s$-stable subspace $W\subset \C^n$ such that $W\cap F_1 =(0)$.
\end{cor}

\begin{proof}
The triviality index of $\E$ is equal to $i(\E)=\sum_{i=1}^s n_i(\kk_1-\kk_i)$.
According to propositions~\ref{GS}, any adjacent
weak solution $\E'$ is given by an $A_s$-stable subspace $W\subset \C^n$.
For any basis $(e)$ of $\C^n$ respecting the flag $\hn$, the bundle $\E'$ has 
type $\k'=\k-T$ where $T_i=0$ when $e_i\in W$ and $T_i=1$ otherwise,
therefore $i(\E')=\sum_{i=1}^n (\max(k_i-T_i)-k_i+T_i)$ where $k_i$ represent the 
elements of $\k$ without multiplicities.
Accordingly, we have $$i(\E)-i(\E')=\sum_{i=1}^n (k_1-T_i-\max(k_i-T_i)).$$
Now, if there exists $i$ such that $k_i=\kk_1$ and $T_i=0$, then $\max(k_i-T_i)=\kk_1$,
thus $i(\E)-i(\E')=\sum_{i=1}^n -T_i <0$ (because we exclude the trivial case $W=\C^n$).
Otherwise we have $\max(k_i-T_i)=\kk_1-1$, and then $i(\E)-i(\E')=\sum_{i=1}^n (1-T_i) >0$.
Therefore $\E'$ exists if and only if there exists $W$ stable under some $A_s$ such that 
$W\cap F_{1}=0$.
\end{proof}

\begin{prop}
Let $\F$ be a Birkhoff-Grothendieck trivialisation of $\D$ at $x\notin \dd$. 
If there exists a flag $F$ in $Y_{\F}$ which is {\em transversal} to $\hn$, and is 
moreover {\em stable} under the action of one of the maps $\psi_s$ for $s\in
\dd$, then the strong Riemann-Hilbert problem has a solution, 
which moreover coincides with $\D$ outside $s$.
\end{prop}

\begin{proof}
Let $F$ be a flag of $Y_{\F}$, which is stable under $\psi_s$. Taking stalks at $x$ of
a $\C$-basis of $F$, we can see the flag $F$ in $D=\D_s/\id_s\D_s$. According to 
lemma~\ref{localHNF}, \ref{localHNFiii}), there exists a Birkhoff-Grothendieck
trivialisation $\E$ of $\D$ at $x$, whose image in $D=\D_s/\id_s\D_s$ is $F$.
Let $(e)$ be a Birkhoff-Grothendieck basis of $\D_s$ with respect to $\E_s$. 
Consequently, its image in $D$ respects the flag $F$. Let $\Upsilon$ be a
Birkhoff form of $\D_s$, and let $(e_{\Upsilon})$ be the $\Upsilon$-basis
of $(e)$. Since the gauge from $(e)$ to $(e_{\Upsilon})$ is tangent to $I$, 
the lattice $M$ induced from $(e_{\Upsilon})$ by the elementary
divisors $K$ of $\E_s$ in $\la$ is also a trivialising
Birkhoff-Grothendieck lattice for $\D$ at $s$. However, the lattice
$M$ is also logarithmic, since by construction it induces in $D$ the
$\psi_s$-stable flag $F$, and moreover sits inside an apartment
that contains the Birkhoff form $\Upsilon$. Hence, the bundle $\D^M$ is both
trivial and logarithmic.
\end{proof}

We have represented the weak solutions to the Riemann-Hilbert problem
as points in a product of subvarieties of stable flags.

\begin{thm}
\label{weak:flags}
Let $\D$ be the Deligne bundle, and $\F$ a Birkhoff-Grothendieck 
trivialisation at an apparent singularity $x\notin \dd$.
The set of weak solutions to the Riemann-Hilbert problem for $\chi$
is parameterised by the set 
$$\RH_{\chi}=\{(F_s,K_s)_{s\in \dd} \tq F_s \in \fl_{\psi_s}(Y), K_s
\in \Z^n(F_s)\},$$
where $Y=\Gamma(X,\F)$ and $\psi_s=\r_s^{\F}\na\in \End_{\C}(Y)$ for $s\in \dd$.
\end{thm}

\subsection{The Type of the Deligne Bundle}
\label{D:type}

The strong version of the Riemann-Hilbert would directly have a
solution if the Deligne bundle were trivial. However, this is not
the case, unless all singular points are apparent, since the exponents
of $\na$ are normalised in such a way
that their sum is non-negative. This means that the type of the
Deligne bundle as a rule is not trivial.
We have seen several ways to characterise this non-triviality. The
{\em type} characterises the isomorphism classes
of holomorphic vector bundles, so it would seem be possible to work
with this sole information. However, we are not in the
right category to do so, since we consider holomorphic bundles with an
embedding in a meromorphic one, denoted with $\V$.
This is the reason for which there are {\em several} trivial bundles
in $\V$. From another point of view, it is not possible
to determine on the sole basis of the sequence $T=(a_1,\ldots,a_n)$,
what the effect of changing the stalk of $\D$ at $x$
will be. Obviously the geometry of the Harder-Narasimhan filtration
will play a decisive role.

\subsubsection{Trivialisations of the Deligne Bundle}
\label{TrivDe}

Let us examine in further detail the case of the Deligne bundle $\D$.
Let us say that $\dt_i$ is an elementary generator of the homotopy group 
$G=\pi_1(X\backslash \dd, z_0)$, if $\dt_i$ is a closed path based at $z_0$,
having winding number $+1$ around the singularity $s_i$ and 0 around the others.
Let $G_i=\chi(\dt_i)$ and $L_i=\frac{1}{2i\pi}\log G_i$ normalised as for
the Deligne lattice. Let $(\sg_i)$ be a basis of the Birkhoff form $\Upsilon_i$ at $s_i$
described in remark~\ref{bf}, such that the connection has locally as
matrix $\Omega_i=L_i\frac{dz}{z}$, on a neighbourhood, say $D_i$ of $s_i$.
On the other hand, let $D_0$ be a neighbourhood of $z_0$, and consider a
basis $(\sg_0)$ of the local Birkhoff form. According to what precedes, 
$(\sg_0)$ is a basis of local {\em $\na$-horizontal sections} of $\D$ over $D_0$.
One can moreover choose this basis in such a way that the monodromy of $(\sg_0)$
around $s_i$ is exactly given by the matrix $G_i$.

Assume now for simplicity that $x\notin \dd$ is the point at infinity 
$\infty\in \P$, and let $\F$ be a trivialisation of $\D$ at $x$.
In a global basis of sections $(e)$ of the bundle $\F$, there exist
matrices $B_i \in \m(\C)$ and a matrix
$$B(z)=B_0+\cdots+B_t z^t\mbox{ and }C_i\in\Gc\mbox{ for }
1\lsl i\lsl p$$ such that the connection has the following matrix
$$\Omega=\left(\sum_{i=1}^p \frac{C_i^{-1}L_iC_i}{z-s_i}+B(z)\right)dz.$$

\begin{rem}
If the bundle $\F$ is moreover logarithmic at $\infty$ -- which
can be achieved, {\it e. g.} by Plemelj's theorem -- then $B=0$ 
and the residue at infinity $L_{\infty}=-\sum_{i=1}^p C_i^{-1}L_iC_i$ 
is semi-simple with integer eigenvalues ({\it ssie}).
At the cost of a (harmless) global conjugation, we can already assume that
$$L_{\infty}=\diag(b_1I_{n_1},\ldots,b_sI_{n_s})\mbox{ with
}b_1<\ldots<b_s.$$ Note that the sequence $\mathcal{B}=(b_1I_{n_1},\ldots,b_sI_{n_s})$
coincides with the elementary divisors of the stalk $\F_{\infty}$ in $\D_{\infty}$.
\end{rem}

\begin{defn}
We say that $(C_1,\ldots,C_p)\in \Gc^p$ is a {\em normalising
$p$-tuple} for $\chi$ if
$\sum_{i=1}^p C_i^{-1}L_iC_i$ is ssie for some (and therefore any)
normalised logarithms
$L_i$ of the generators $\chi(\ga_i)$ of the monodromy group.
\end{defn}

Normalising $p$-tuples always exist. Putting $t$ as
the coordinate $1/z$ at infinity,
the Taylor expansion of $\na$ at $x=\infty$ has then the follwing nice
expression
\begin{equation}\label{Tinf}
\Omega=-\sum_{k\gsl 0}
\sum_{i=1}^p s_i^k\ti{L}_i t^k\frac{dt}{t}\mbox{ with
}\ti{L}_i=C_i^{-1}L_iC_i.
\end{equation}

We have thus reduced the computation of the type of the Deligne bundle
to the computation of the matrices $C_i$ (the so-called {\em connection matrices}, because
they connect the different local expressions of $\na$ on the local 
Birkhoff forms).
It is however well known that the computation of the connection matrices
is difficult. Any other trivialisation of $\D$ at infinity is given by a {\em monopole gauge}~(\cite{IlYa}), namely
a unimodular polynomial matrix $\Pi\in \G(\C[z])$, that is, a matrix satisfying
$$\Pi=P_0+P_1z+\cdots +P_kz^k\mbox{ such that }\det \Pi(z)={\rm
cst}\in \C^*.$$

\begin{prop}
Given a family of points $s_1,\ldots,s_p\in \C$ and invertible
matrices $C_1,\ldots,C_p\in \Gc$ all having
the same determinant, there exists a monopole gauge $\Pi\in \G(\C[z])$
such that $\Pi(s_i)=C_i$ for
$1\lsl i\lsl p$.
\end{prop}

\begin{proof}
The group $\SL_n(R)$ on a ring is generated by transvections
$T_{ij}(\lb)=I+\lb E_{ij}$ where $\lb\in R$
and $E_{ij}$ is the $(i,j)$ element of the canonical basis of the
vector space $\mathfrak{gl}_n$. Factoring
out the common value of $\det C_i$, one can assume that the matrices
$C_i$ are in $\SL_n(\C)$, and that they
appear as a product of transvections. At the cost of introducing the
trivial transvections $T_{ij}(0)=I$,
one can even assume that all are factored as a product of the {\em
same} transvections with different parameters
$$C_i=T_1(\mu_1^i)\cdots T_s(\mu^i_s)\mbox{ with }\mu^i_t\in \C.$$
Define then $\lb_k\in\C[z]$ such that $\lb_k(s_i)=\mu_k^i$ for
$1\lsl i\lsl p$. By construction, the product
$\ti{\Pi}=T_1(\lb_1)\cdots T_s(\lb_s)\in \SL_n(\C[z])$ indeed
interpolates
the matrices $C_i$ at the points $s_i$. The general case is obtained
by multiplying $\ti{\Pi}$ by the common value
of $\det C_i$.
\end{proof}

As a consequence of this result, one can find a trivialisation $\E$ at
infinity of the Deligne bundle such that the
residues of the connection $\na$ are expressed in a basis of
$Y=\Gamma(X,\E)$ as the actual matrices
$L_i$ (and not {\em conjugated} to them). Although the
point at infinity of $\E$ is still an apparent singularity,
we have no control on the Poincar\'e rank of $\na$ at $\infty$.

The results of this section also hold
(with the adequate modifications) if the apparent singularity is
assumed to be located at $z_0\not\in\dd\cup\{\infty\}$.
We will refer to the trivialisation $\E$ as an {\em adapted
trivialisation of $\D$ at $z_0$}. 

\begin{rem}
We know that there exists a family of invertible matrices
$(C_i)$ such that $\sum_{i=1}^p C_i^{-1}L_iC_i$ is semi-simple with
integer eigenvalues and that these eigenvalues are equal to the type of the Deligne bundle.
This raises two questions:
\begin{enumerate}
\item Does there exist a logarithmic trivialisation of $\D$ for any such family
$(C_i)$?
\item If there exist several families with this property, how to recognize those
that indeed give the type of the Deligne bundle? 
\end{enumerate}
\end{rem}

\subsection{Reducibility of the Monodromy Representation}
\label{red}

We establish now an improvement of a result of Bolibrukh~\cite{BoA} (prop. 4.2.1, p. 84).

\begin{prop}
\label{Boli:irred}
If the representation $\chi$ is irreducible, then for any weak solution 
$\E\in \RH_{\chi}$, the type $(k_1,\ldots,k_n)$ of $\E$
satisfies $k_i-k_j\leq p-2$.
\end{prop}

\begin{proof}
Assume here for simplicity that $x=\infty\notin \dd$, and consider
again the setting of section \ref{TrivDe}. Let $\E$ be any weak solution to Riemann-Hilbert, 
and $\F$ be a logarithmic Birkhoff-Grothendieck trivialisation
of $\E$ at $x$. Let $K=(k_1,\ldots,k_n)$ be the type of $\E$.
In a basis $(e)$ of global sections of $\F$, there exist constant matrices $\ti{L}_a$
for $a\in \dd$ such that the connection $\na$ has in $(e)$ the following matrix

\begin{equation}
\label{FS}
\Omega=\sum_{a\in\dd} \frac{\ti{L}_a}{z-a}dz=-\frac{dt}{t}\sum_{k\gsl 0}\Omega_k t^k
\mbox{ with }\Omega_k=\sum_{a\in\dd} a^k\ti{L}_a\mbox{ and }t=\frac{1}{z}.
\end{equation}

By lemma~\ref{Deligne:hol}, the shearing $t^{-K}$ suppresses the singularity at $x$, since the
basis $t^{-K}(e)$ spans the Deligne lattice. As a consequence,
$\ti{\Omega}=t^{K}\Omega(t)t^{-K}+
K\frac{dt}{t}$ must satisfy $v(\ti{\Omega})\geq 0.$ Therefore, the
residue matrix $B=-\sum_{a\in\dd}\ti{L}_a$ of $\Omega$ at $x$ is diagonal and equal to $-K$. 
We can assume further that $$B=\diag(b_1 I_{n_1},\ldots,b_s I_{n_s})\mbox{ with }b_1=-k_1<\cdots <b_s=-k_s$$ 
where $(k_1I_{n_1},\ldots,k_sI_{n_s})$ represents the type of $\E$ with multiplicities. 
Partition any matrix $M$ according to the eigenvalue
multiplicities of $B$, as $(M_{\ell,m})$ for $1\leq \ell,m \leq s$.
Then the matrix of the connection can be rewritten by blocks as
$$\ti{\Omega}_{\ell,m}=\Omega_{\ell,m}t^{k_{\ell}-k_m}+K\frac{dt}{t}=\left(-\sum_{j\geq
0}\Omega^{(j)}_{\ell, m}
t^{j+k_{\ell}-k_m}+\delta_{\ell,m}k_{\ell}I_{n_{\ell}}\right)\frac{dt}{t}.$$
For each $(\ell,m)$ block, this series must have strictly positive
valuation. 
The sum $\sum_{a\in\dd}\ti{L}_a=K$ imposes conditions on all blocks of
the residues $\ti{L}_a$, while when $\ell>m$ we get the following equations.
\begin{equation}
\Omega^{(j)}_{\ell,m}=\sum_{a\in\dd} a^j(\ti{L}_a)_{\ell,m}=0\mbox{ for }0\leq j\leq
k_m-k_{\ell}\mbox{ when }\ell>m.\label{e2}
\end{equation}

For a fixed pair $(\ell,m)$, let $k=\max(0,k_m-k_{\ell})$, and let $X_i\in
\C^{n_{\ell}\times n_m}$ be the $(\ell,m)$-block of the matrix $\ti{L}_{s_i}$,
for $1\leq i\leq p$. For $1\leq \alpha\leq n_{\ell}$ and $1\leq
\beta\leq n_{m}$, let $v_{\alpha,\beta}\in \C^p$ be the
vector constructed by taking the coefficient of index $(\alpha,\beta)$
of $X_i$, for $1\leq i \leq p$.
Then, the equations (\ref{e2}) can be reformulated as
$$v_{\alpha,\beta} \in \ker M_k(\underline{s})\mbox{ where }
M_k(\underline{s})=\left(
\begin{array}{ccc}
1 & \cdots & 1 \\
s_1 & \cdots & s_p \\
\vdots &    & \vdots \\
s_1^k & \cdots & s_p^k
\end{array}
\right).
$$

The matrix $M_k(\underline{s})$ is an upper-left submatrix of a
Vandermonde matrix with coefficients
$$\underline{s}=(s_1,\ldots,s_p)\in \C^p\backslash
\bigcup_{i\neq j}\{x_i\neq x_j\}.$$ Since all the
$s_i$ are distinct, this matrix has always full rank. In particular,
as soon as $k_m-k_{\ell}\geq p-1$, it
has a null kernel, and so all the blocks $X_i$ are zero. Due to the
ordering of the $k_i$, we also have
$k_{m'}-k_{\ell'} \gsl p-1$ for $m'\lsl m$ and $\ell'\gsl \ell$, thus all matrices
$\ti{L}_a$ have a lower-left common zero block.
This means that the representation $\chi$ is reducible.
\end{proof}

\subsection{Testing the Solubility of the Riemann-Hilbert Problem}
\label{RH:sol}

In this section, we apply the results of this paper to the
experimental investigation of the solubility of the Riemann-Hilbert
problem. We present two ways to search the space of weak solutions,
which are completely effective (up to the known problem of connection
matrices): one that follows paths of adjacent logarithmic lattices,
based on lemma~\ref{adj:log}, the other that uses the characterisation
as stable flags given in proposition~\ref{bij:lattices}. Note that, if
any (not necessarily logarithmic) trivial holomorphic bundle of the meromorphic solution to
Riemann-Hilbert is explicitly given, the procedures that we present,
coupled with classical Poincar\'e rank reduction methods,
implemented on a computer algebra system, allow to make the actual
computations. We however do not know if this bypasses the problem of
the connection matrices.

Let $\D$ be the Deligne bundle of the representation $\chi$. Let
$x\notin \dd$, and consider a logarithmic Birkhoff-Grothendieck
trivialisation $\F$ of $\D$ at $x$. Let $Y=\Gamma(X,\F)$ and choose a
basis $(\sg)$ of $Y$ in which the residue matrix at $x$ is equal to the diagonal
that represents the type of $\D$
$$\mat(\r_{x}^{\F}\na,(\sg))=-\k=\diag(-k_1I_{n_1},\ldots,-k_{s}I_{n_s})\mbox{ where }k_1>\cdots>k_s.$$
In the basis $(\sg)$, the connection has a matrix of the
form (\ref{FS}), and the Harder-Narasimhan filtration is expressed as the
flag $\hn_Y$ of signature $(n_1,\dots,n_s)$ of $Y$.
Let $V=\Gamma(X,\V)$ be the $\mf{K}$-vector space of meromorphic sections of $\V$, where
$\mf{K}=\Gamma(X,\mx)$ is the field of meromorphic functions on $X$. 

For $s\in \dd$, let $t$ be a coordinate at $x$ with divisor $(t)=x-s$, and $(\hsg)=t^{-\k}(\sg)$.
Recall that $t^{-1}$ is a coordinate at $s$. For clarity's sake, we will put $t_x=t$ and $t_s=t^{-1}$
when we are dealing with local sections.
Let $\ti{\F}=t_s(\F)$ be the transport of $\F$ at $s$ and $\ti{Y}=\Gamma(X,\ti{\F})$.
We regard $Y$ and $\ti{Y}$ as sub-$\C$-vector spaces of $V$, spanned respectively by 
the $\mf{K}$-bases $(\sg)$ and $(\hsg)$ of $V$. 
The relation $(\hsg)=t^{-\k}(\sg)$ induces a well-defined fixed isomorphism between $Y$ and $\ti{Y}$.
\begin{enumerate}[{\bf Claim} 1:]
\item The trivial bundle $\ti{\F}$ is a Birkhoff-Grothendieck trivialisation
of $\D$ at $s$.
\item The flag $\hn_{\ti{Y}}$ is the flag of signature $(n_1,\dots,n_s)$ spanned by $(\hsg)$.
\item The germ $(\sg_s)$ of the global basis of $Y$ at
$s$ is a local basis of $\D_s$. 
\end{enumerate}

Indeed, we have the two dual schematic representations,
where $(\sg_x):\E_x$ means that $(\sg)$ is a local basis of $\E$ at $x$ and $(\sg):Y$ means that 
$(\sg)$ is a global basis of the form $Y$
$$\appname{(\hsg_x):\D_x}{(\sg):Y}{t_x^{\k}}\mbox{ and }\appname{(\sg_s):\D_s}{(\hsg):\ti{Y}}{t_s^{\k}}.$$

\subsubsection{Adjacent Lattices}
\label{RH:sol1.1}

In this section, we consider more generally a weak solution $\E\in \RH_{\chi}$.
In the following proposition, we describe a procedure which allows to read off
at an apparent singularity $x\notin \dd$, fixed once and for all, the effect
on the weak solution $\E$ of a change of logarithmic adjacent 
lattice at any singularity $s\in \dd$. More precisely, let $(\sg)$ be a global basis
of a logarithmic Birkhoff-Grothendieck trivialisation of $\E$ at $x$,
and $\Omega$ the matrix in Fuchsian form~(\ref{FS}) of the connection $\na$ in $(\sg)$,
whose residue at $x$ gives precisely the type of $\E$. Let $M$ be a logarithmic lattice at $s$ 
that is adjacent to $\E_s$. We determine explicitly a gauge transform $\Pi_M$
which is a monopole at $x$, such that $\Omega_{[\Pi_M]}$ has again Fuchsian form~(\ref{FS}).
From its semi-simple residue at $x$ we read directly the type of the
modified bundle $\E^M$, equal to the eigenvalues, and the
Harder-Narasimhan filtration of $\E^M$, spanned by the eigenspaces ordered by increasing values.

This procedure is completely effective once the connection matrices $C_s$
that relate the {\em local} residue matrices $L_s=\frac{1}{2i\pi}\log G_s$ in the 
Birkhoff form at $s$ and the {\em global} residue matrices $\ti{L}_s=C_s^{-1}L_sC_s$
in the basis $(\sg)$, have been determined.

Let $M$ be a lattice at $s$ that is adjacent to $\E_s$. This
lattice is uniquely characterised by its image $W=M/\id_s\E_s$, 
that can be seen as a sub-$\C$-vector space $W\subset Y$.
It is logarithmic if and only if $W$ is
stable under the map $\r_s^{\E}\na$. 

According to proposition~\ref{GS}, a Birkhoff-Grothendieck trivialisation of $\E^M$
is obtained from a basis of $\E_s$ that simultaneously
respects the space $W$ and the flag $\hn$. Moreover, we can choose
$(\vep)$ in the $\Gc$-orbit of $(\sg)$.

\begin{enumerate}[{\bf Claim} 1:]
\setcounter{enumi}{3}
\item There exists a basis $(\vep)$ of $Y$ such that $t_s^{\k}(\vep)$ spans a
Birkhoff-Grothendieck trivialisation of both $\E$ and $\E^M$ at $s$.
\item The matrix $P\in\Gc$ of the basis change from
$(\sg)$ to $(\vep)$ is $(-\k)$-parabolic.
\item The gauge $t_s^{-\k}Pt_s^{\k}=t_x^{\k}Pt_x^{-\k}$ is a monopole at $s$ and an element of $\G(\h_x)$.
\end{enumerate}

$$
\xymatrix{
(\hsg):\ti{Y} \ar[r]^-{t_s^{-\k}Pt_s^{\k}}& (\ti{\vep}):Y' \\
(\sg):\E_s \ar[u]^-{t_s^{\k}} \ar[r]^-{P} \ar[d]^-{\pi} & (\vep):\E_s
\ar[u]^{t_s^{\k}} \ar[r]^-{t_s^{T}} \ar[d]^-{\pi} & (\sg'):M \ar@/_.5pc/[ul]_-{t_s^{\k-T}}\\
E=\E_s/\id_s\E_s \ar[r]^-{P} & W
}
$$
\begin{enumerate}[{\bf Claim} 1:]
\setcounter{enumi}{6}
\item The basis $(\sg')$ generates $M$ at $s$ and $\E_y$ at $y\neq x$.
\item The trivial bundle $\F'$ spanned by $(\sg')$ is a Birkhoff-Grothendieck trivialisation of $\E^M$ at $x$.
\item The gauge transform from $(\sg)$ to $(\sg')$ is $Pt_s^T=Pt_x^{-T}$.
\item The Harder-Narasimhan filtration of $\E^M$ is given by the
flag of $Y'$ spanned by $(\sg')$ according to $\k-T$.
\end{enumerate}
Indeed, the last arrow on the right implies that at $x$, we have $$\xymatrix{(\ti{\vep}_x):\E_x=\E^M_x \ar[r]^-{t_x^{\k-T}} & (\sg'):Y' }
\mbox{ where }Y'\subset V\mbox{ is spanned over }\C\mbox{ by }(\sg').$$
Therefore the type of $\E^M$ is, as expected, equal to $\k-T$.

\begin{prop}
\label{gauge}
Assume that $\dd\subset \C$ and $x=\infty$. Let $\E\in\RH_{\chi}$ be a weak solution
to the Riemann-Hilbert problem. Let the connection $\na$ have a matrix
$\Omega$ of the form~(\ref{FS}) in a basis $(\sg)$ of a logarithmic Birkhoff-Grothendieck
trivialisation $\F$ of $\E$ at $x$.
Then, for any $\ti{L}_s$-stable subspace $W_s$
of $\C^n$, there exists a computable monopole gauge $\Pi\in
\G(\C[z])$, a constant matrix $P_0\in \Gc$ and a diagonal matrix $T$
with only $0,1$ elements such that
$\Omega_{[P_0(z-s)^T\Pi]}$ has again a form~(\ref{FS}) corresponding to
the modification $\E^M$, where $M$ is the lattice of $\V_s$ adjacent to $\E_s$
canonically defined by $W_s$.
\end{prop}

\begin{proof}
We identify $\Gamma(X,\F)$ with $\C^n$ by means of the
basis $(\sg)$. The residue of $\na$ at $s$ is then equal to the matrix
$L=\ti{L}_s$ of formula~(\ref{FS}). A logarithmic adjacent lattice $M$
is uniquely defined by an $L$-stable subspace $W\subset\C^n$. Let
$(\vep)$ be a basis respecting both $W$ and the Harder-Narasimhan flag
$F$, and let $P\in\Gc$ be the basis change from $(\sg)$ to
$(\vep)$. Assume for simplicity that we have ordered the vectors
$\vep_1,\ldots,\vep_n$ in such a way that if $\vep_i\in F_k\cap W$ and
$\vep_{i+1}\notin W$ then $\vep_{i+1}\notin F_k$. Let
$T=\diag(t_1,\ldots,t_n)$ be the diagonal integer matrix defined by
$t_i=1$ if and only if $\vep_i\notin W$. With the simplifying
assumption, the type of $\E^M$ is equal to $K-T$, including the
ordering condition, and the Harder-Narasimhan filtration is exactly
obtained by putting together the groups of vectors corresponding to
equal values of $K-T$. Therefore the basis $(\sg')=(z-s)^{K-T}(\vep)$
spans a Birkhoff-Grothendieck $\F'$ trivialisation of $\E^M$ at $s$,
and it is simultaneously a global basis of $V$. The transport
$t_x(\F')$ is again a Birkhoff-Grothendieck trivialisation of $\E^M$
at $x$, but it needs not be logarithmic anymore. 
Since $\E$ is a weak solution, we have $\E_x=\E_x^{M}=\D_x$.
Therefore, there exists a lattice gauge transformation
$P=I +P_1t_x + P_2 t_x^2+\cdots$ which sends the basis $(\sg')$
into its $\Upsilon$-basis $(\vep')$, where $\Upsilon$ is the Birkhoff form at $x$. 
The lattice $M'$ spanned by $t_x^{K-T}(\sg')$ is then necessarily logarithmic, according to 
proposition~\ref{poly:form}. We can effectively determine $M'$ by truncating the gauge $P$ at order
$d(M',D)-1= k_n-k_1-2 $, and then applying Gantmacher's classical recursive formul{\ae}~(\ref{G}). 
Then, the permutation lemma yields a monopole gauge
transform $\Pi$ at $x$ so that the resulting
trivialisation $\ov{\F}$ is both Birkhoff-Grothendieck and
logarithmic. In this last basis, the connection has again a
form~(\ref{FS}), where the spectrum of the residue at $x$ gives the
type of the modified logarithmic bundle $\E^M$.
\end{proof}

It results from proposition~\ref{weak:sol} that iterated 
applications of this procedure will describe the set of all weak solutions to the Riemann-Hilbert problem,
and the strong problem will be solvable if under the orbit of these transformations, one of the bundles
$\ov{\F}$ has a 0 residue at $x$.

\subsubsection{The General Case}
\label{RH:sol1.2}

For the general case, we start with the Deligne bundle $\D$, for we 
only have the complete description of the local logarithmic lattices 
from the Deligne lattice.

According to the description given in proposition~\ref{bij:lattices}, any
logarithmic lattice $N\in\ra_s$ is given by an admissible pair
$(F,T)$ where $F$ is a $\r_s^{\D}\na$-stable flag. If we put us in
the situation of section~\ref{RH:sol1.1}, and consider a logarithmic
Birkhoff-Grothendieck trivialisation $\F$ of $\D$ at $x$, and identifying
$\Gamma(X,\F)$ to $\C^n$ by means of the basis $(\sg)$, then the
flag $F$ can be viewed as a flag in $\C^n$ stable under the matrix
$\ti{L}_s$. In order to actually construct the lattice $N$, one should
in principle reach first a Birkhoff form $\Upsilon_z$ in $\la=\D_s$.
We know from lemma~\ref{form:dist} that if we put $d=\max(t_i-t_j)$, 
a gauge $P$ of $z$-degree $d-1$ is already sufficient, as remarked in the
proof of proposition~\ref{gauge}. Let $(\vep)$ the basis obtained by $P$.
In the apartment spanned by the basis $(\vep)$ of $\la$ there is 
a Birkhoff-Grothendieck trivialisation $\tM$ of $\la$, as shown in the following scheme.

$$
\xymatrix{
\la:(\sg) \ar[r]^-{t^{\k}} \ar[d]^-{P} & Y_M \ar[dr]^-{\Pi} &  \\
\la:(\vep) \ar[dr]_-{t^{T}} \ar[r]^-{t^{\k_{\tau}}} & \tM \ar[r]^-{\ti{P}} & Y_{\tM} \\
& N \ar[u]_{t^{\k_{\tau}-T}} &
}
$$

Here we cannot avoid the permutation $\tau\in S_n$, because we can't
ensure that the Birkhoff gauge $P$ satisfies the principal minors condition from the permutation lemma:
the constant term $P_0$ sends the basis $(\ov{\sg})$ onto a basis that respects both
the Harder-Narasimhan flag of $\la$ and the flag $F$, but not as an {\em ordered} basis. 
Actually, the permutation $\tau\in S_n$ is
the label of the Schubert cell of $\Gc$ that contains the matrix $P_0$.

Although we can explicitly determine the diagonal
$\k_{\tau}-T$, there is no reason that these integers give the type of
$\D^M$, nor that $\tM$ is a Birkhoff-Grothendieck trivialisation of
$N$. The lattice $\tM$ is nevertheless a trivialising lattice,
therefore it is possible to compute a Birkhoff-Grothendieck
trivialisation $M'$ by means of the algorithm described in
section~\ref{BG:algo}. Indeed, the gauge $\ti{P}$, and therefore the monopole $\Pi$, have polynomial coefficients
that can be effectively computed.

We would also like to note that very recently and independently, in the arXiv
paper~\cite{PB}, P. Boalch has taken a similar view on local
logarithmic lattices, in terms of stable filtrations and Bruhat-Tits buildings.


\end{document}

%% file: Permutation.tex
The proof proceeds by induction, using the following simple lemma.

\begin{lemm}
\label{iter:step}
Let $k\lsl n$ and $T=\left(\begin{array}{cc} I_k & 0 \\ 0 & 0_{n-k} \end{array} \right)$.
Let $H=\left(\begin{array}{cc} A & B \\ C & D \end{array} \right)\in \G(\C[[t]])$ be a lattice gauge matrix,
decomposed as a $2\times 2$-block matrix according to the blocks of $T$. If $\det A(0)\neq 0$, then there exists
a monopole gauge matrix $\Pi=\left(\begin{array}{cc} I_k & t^{-1}\ti{\Pi} \\ 0 & I_{n-k} \end{array} \right)$
with $\ti{\Pi}$ a constant matrix, such that $\ti{H}=t^{-T}Ht^T\Pi$ is a lattice gauge matrix, that is $\ti{H}\in\Go$.
\end{lemm}

\begin{proof}
Put for simplicity $M_0=M(0)$ for a holomorphic matrix $M$.
One checks that putting $\ti{\Pi}=-A_0^{-1}B_0$, we have
$$\ti{H}=t^{-T}Ht^T\Pi=\left(\begin{array}{cc} A & \ti{B} \\ tC & \ti{D} \end{array} \right),$$
where $\ti{B}=t^{-1}(B+A\ti{\Pi})$ and $\ti{D}=D+C\ti{\Pi}$.
By construction, the residue of $\ti{B}$ is equal to $B_0-A_0A_0^{-1}B_0=0$, hence $\ti{B}$ is holomorphic;
therefore $\ti{H}$ also is. To check that $\ti{H}\in \Go$, it is sufficient to
check the invertibility of 
$$\ti{H}_0=\left(\begin{array}{cc} A_0 & \ti{B}_0 \\ 0 & D_0-C_0A_0^{-1}B_0 \end{array} \right).$$
By assumption $A_0$ is invertible, and it is a simple exercise in linear algebra to show that 
$D-CA^{-1}B$ is invertible when $\left(\begin{array}{cc} A & B \\ C & D \end{array} \right)\in\Gc$ is.
\end{proof}

Note that the upper-left block of $H$ appears unchanged in $\ti{H}$. Note also that 
$\ov{H}=t^T\Pi=\left(\begin{array}{cc} tI_k & \ti{\Pi} \\ 0 & I_{n-k} \end{array} \right)$. 
Geometrically, we can summarize the construction of lemma~\ref{iter:step} as 
the following scheme.
$$
\xymatrix{
\la \ar[r]^-{H}\ar[d]_-{t^T} & \la \ar[r]^-{t^T}\ar[d]^-{\ov{H}} & M \ar[dl]^-{\Pi}  \\
\tilde{\la}  \ar[r]^-{\tilde{H}} & \tilde{\la} &
}
$$

We only need a small technical lemma before giving the actual proof of the permutation lemma.
Let $K$ denote an integer sequence $(\underbrace{k_1,\ldots,k_1}_{n_1
\mbox{ \scriptsize{times}}},\ldots,\underbrace{k_s,\ldots,k_s}_{n_s
\mbox{ \scriptsize{times}}})$ with $k_i>k_{i+1}$. We say that a matrix $H$ is {\em strongly $K$-parabolic} if 
it has the following form
$$H=\left(\begin{array}{ccc} t^{k_1}I_{n_1} & \cdots
  & P_{ij} \\  & t^{k_2}I_{n_2} & \vdots \\
0 &  &  t^{k_s}I_{n_s} \end{array} \right),$$ where $P_{ij}$ is
  a $n_i\times n_j$ polynomial matrix satisfying $\deg P_{ij}< k_i$ and
  $v(P_{ij})\gsl k_j$.

\begin{lemm}\label{max:parabolic}
Let $H$ be strongly $K$-parabolic, and let 
$H'=\left(\begin{array}{cc} tI_m & \ti{\Pi} \\ 0 & I_{n-m} \end{array} \right)$,
where $\ti{\Pi}$ is a constant matrix and $m\lsl n_1$.
Then the product $HH'$ is strongly $K'$-parabolic, where
$K'=(\underbrace{k_1+1,\ldots,k_1+1}_{m\mbox{ \scriptsize{times}}},\underbrace{k_1,\ldots,k_1}_{n_1-m\mbox{ \scriptsize{times}}},\ldots,\underbrace{k_s,\ldots,k_s}_{n_s\mbox{ \scriptsize{times}}})$.
\end{lemm}

\begin{proof}
Let $\ov{K}=(\underbrace{k_2,\ldots,k_2}_{n_2\mbox{ \scriptsize{times}}},\ldots,
\underbrace{k_s,\ldots,k_s}_{n_s\mbox{ \scriptsize{times}}})$. The matrix $H$ can be written as
$H=\left(\begin{array}{cc} t^{k_1}I_{n_1} & P \\ 0 & \ov{H} \end{array} \right)$, where $\ov{H}$ is
strongly $\ov{K}$-parabolic, and $P=\left(\begin{array}{ccc} P_2 & \cdots & P_s \end{array} \right)$
where the blocks $P_i$ satisfy $\deg P_i<k_1$ and $v(P_k)\gsl k_i$. Then, if $m=n_1$, the product $HH'$ is simply
$$HH'=\left(\begin{array}{cc} t^{k_1+1}I_{n_1} & t^{k_1}\ti{\Pi}+P \\ 0 & \ov{H} \end{array} \right).$$
Otherwise, we split the matrices in $3\times 3$-blocks, as 
\begin{eqnarray*}HH'&=&\left(\begin{array}{ccc} t^{m}I_{m} & 0 & P_1 \\ 
0 & t^{k_1}I_{n_1-m} & P_2 \\
0 & 0 & \ov{H} \end{array} \right) \left(\begin{array}{ccc} tI_{m} & \ti{\Pi}_1 & \ti{\Pi}_2 \\ 
0 & I_{n_1-m} & 0 \\
0 & 0 & I_{n-n_1} \end{array} \right)\\
&=& \left(\begin{array}{ccc} t^{k_1+1}I_{m} & t^{k_1}\ti{\Pi}_1 & t^{k_1}\ti{\Pi}_1+P_1 \\ 
0 & t^{k_1}I_{n_1-m} & P_2 \\
0 & 0 & \ov{H} \end{array} \right)
\end{eqnarray*}
In both cases, we see that the product $HH'$ is strongly $K'$-parabolic as requested.
\end{proof}

\begin{proof}[Proof of lemma~\ref{lem:per}]
Assume for simplicity that $K=\diag(k_1I_{n_1},\ldots,k_sI_{n_s})$ is written by blocks, and that 
$k_1>k_2>\ldots >k_s$. Then there exist $m=k_1-k_s$ matrices $T_1,\ldots,T_m$ of the type 
$T_i=\left(\begin{array}{cc} I_{b_i} & 0 \\ 0 & 0_{n-b_i} \end{array} \right)$,
where every $b_i$ is equal to some $n_1+\cdots+n_{t_i}$ for some decreasing sequence $t_i$, 
such that $K=T_1+\cdots+T_m$. Secondly, assume that all left-upper square blocks of $H_0$ of sizes 
$b_i$ are invertible. Letting $H=H_1$, according to lemma \ref{iter:step}, 
there exists a sequence of monopole matrices
$\Pi_i=\left(\begin{array}{cc} I_{b_i} & t^{-1}\ti{\Pi}_i \\ 0 &
  I_{n-{b_i}} \end{array} \right)$ with a constant matrix $\ti{\Pi}_i$,
and a sequence of lattice gauge transforms $H_i\in \Go$ such that
\begin{equation}\label{1}
H_{i+1}=t^{-T_i}H_it^{T_i}\Pi_i.
\end{equation}
Let $\ov{H}_i=t^{T_i}\Pi_i=\left(\begin{array}{cc} tI_{b_i} & \ti{\Pi}_i \\ 0 &
  I_{n-{b_i}} \end{array} \right)$. It follows from lemma~\ref{max:parabolic} 
that $\ov{H}=\ov{H}_1\cdots\ov{H}_m$ is strongly $K$-parabolic.
It follows then, as a remarkable consequence, that the diagonal matrix $t^K$ can be both factored
from the matrix $\ov{H}$ both on the left as $\ov{H}=t^K\Pi$ with a
monopole matrix $\Pi$, and simultaneously from the right as
$\ov{H}=Pt^K$ with a lattice gauge $P\in\Go$. Since
  $t^KH_{d+1}=H\ov{H}$ holds, we get on the one hand that $t^{-K}H^{-1}t^KH_{m+1}=\Pi\in\G(\C[t^{-1}])$ 
as required for the permutation lemma. However, and this was not stated in~\cite{Bo1} or \cite{IlYa},
we also have the following relation $t^K\Pi=Pt^K$, which yields the second claim.
\end{proof}

\begin{rem}\label{Pi:para}
It results from the previous proof that the monopole gauge $\Pi$ is block-upper-triangular according to 
$\k$, and that its block matrices $\Pi_{ij}$ satisfy $$k_j-k_i \lsl v(\Pi_{ij})\lsl \deg \Pi_{ij} \lsl 0.$$ 
\end{rem}

As stated in~\cite{IlYa}, one can assume that $\sg =\idd$ if all
leading principal minors of $P$ are holomorphically invertible
(which can always be ensured by a permutation of the columns of $P$).